\documentclass[etterpaper,11pt,reqno]{amsart}

\usepackage{amsmath,amssymb,amscd,amsthm,amsxtra}
\usepackage{ esint }

\usepackage[implicit=true]{hyperref}
\usepackage{bbm}

\usepackage{tikz}
\usetikzlibrary{matrix} 

\usepackage{enumerate}

\allowdisplaybreaks[2]

\usepackage{color}

\newtheorem{theorem}{Theorem} [section]

\newtheorem{lemma}[theorem]{Lemma}
\newtheorem{proposition}[theorem]{Proposition}
\newtheorem{remark}[theorem]{Remark}

\newtheorem{definition}[theorem]{Definition}

\DeclareMathOperator*{\intt}{\int}

\DeclareMathOperator{\med}{med}

\newcommand{\I}{\hspace{0.5mm}\text{I}\hspace{0.5mm}}
\newcommand{\II}{\text{I \hspace{-2.8mm} I} }

\newcommand{\noi}{\noindent}
\newcommand{\Z}{\mathbb{Z}}
\newcommand{\R}{\mathbb{R}}
\newcommand{\C}{\mathbb{C}}
\newcommand{\T}{\mathbb{T}}

\let\Re=\undefined\DeclareMathOperator*{\Re}{Re}
\let\Im=\undefined\DeclareMathOperator*{\Im}{Im}

\let\P= \undefined
\newcommand{\P}{\mathbf{P}}

\newcommand{\E}{\mathbb{E}}

\newcommand{\N}{\mathcal{N}}
\newcommand{\NB}{\mathbb{N}}

\newcommand{\FL}{\mathcal{F}L} 

\newcommand{\Gg}{\mathcal{G}}

\newcommand{\al}{\alpha}
\newcommand{\be}{\beta}

\newcommand{\eps}{\varepsilon}

\newcommand{\g}{\gamma}
\newcommand{\G}{\Gamma}
\newcommand{\ld}{\lambda}

\newcommand{\s}{\sigma}

\newcommand{\ft}{\widehat}
\newcommand{\Ft}{{\mathcal{F}}}

\newcommand{\wb}{\overline}

\newcommand{\embeds}{\hookrightarrow}

\newcommand{\nbar}{\overline{n}}
\newcommand{\mbar}{\overline{m}}
\newcommand{\conj}[1]{\overline{#1}}

\renewcommand{\l}{\ell}

\newcommand{\ges}{\gtrsim}

\newcommand{\jb}[1]
{\langle #1 \rangle}

\usepackage{bbm}

\DeclareMathOperator{\Id}{Id}

\numberwithin{equation}{section}
\numberwithin{theorem}{section}

\usepackage{todonotes}

\begin{document}
	\baselineskip = 14pt

	\title[A remark on the well-posedness of mKdV in Fourier-Lebesgue spaces]
	{A remark on the well-posedness of the modified KdV equation in the Fourier-Lebesgue spaces}

	\author[A.~Chapouto]
	{Andreia Chapouto}

	\address{
		Andreia Chapouto\\ Maxwell Institute for Mathematical Sciences
		and 
		School of Mathematics\\
		The University of Edinburgh\\
		and The Maxwell Institute for the Mathematical Sciences\\
		James Clerk Maxwell Building\\
		The King's Buildings\\
		Peter Guthrie Tait Road\\
		Edinburgh\\ 
		EH9 3FD\\United Kingdom} 
	
	\email{andreia.chapouto@ed.ac.uk}



	
	\begin{abstract}
		We study the complex-valued modified Korteweg-de Vries equation (mKdV) on the circle. We first consider the real-valued setting and show global well-posedness of the (usual) renormalized mKdV equation in the Fourier-Lebesgue spaces. 
		
		In the complex-valued setting, we observe that the momentum plays an important role in the well-posedness theory. In particular, we prove that the complex-valued mKdV equation is ill-posed in the sense of non-existence of solutions when the momentum is infinite, in the spirit of the work on the nonlinear Schr\"odinger equation by Guo-Oh (2018). This non-existence result motivates the introduction of the second renormalized mKdV equation, which we propose as the correct model in the complex-valued setting outside of $H^\frac12(\T)$.
		Furthermore, imposing a new notion of finite momentum for the initial data, at low regularity, we show existence of solutions to the complex-valued mKdV equation. In particular, we require an energy estimate, from which conservation of momentum follows.

	\end{abstract}

	
	\vspace*{-10mm}

	\maketitle
	%
	

	\tableofcontents

\section{Introduction}
\subsection{The modified Korteweg-de Vries equation}
In this paper, we study the Cauchy problem for the complex-valued modified Korteweg-de Vries equation (mKdV) on the one-dimensional torus $\T=\R / (2\pi\Z) $:
\begin{align}
	\begin{cases}
		\partial_t u + \partial_x^3 u = \pm |u|^2 \partial_x u, \\
		u\vert_{t=0} = u_0,
	\end{cases}
	\quad (t,x)\in\R\times\T.
	\label{mkdv}
\end{align}
The mKdV equation \eqref{mkdv} has been extensively studied from both the theoretical and applied points of view. We will pursue a harmonic analytic approach to study the well-posedness of \eqref{mkdv} in the Fourier-Lebesgue spaces (see \eqref{FL}). Let us first go over the local-in-time results in \(L^2\)-based Sobolev spaces. In \cite{BO2}, Bourgain introduced the Fourier restriction norm method, utilizing the $X^{s,b}$-spaces (see \eqref{xsb}), and proved the local well-posedness in $H^s(\T)$, for $s\geq \frac12$, of the \textit{first renormalized}\footnote{The equation \eqref{renorm} is usually referred to as the \textit{renormalized} mKdV equation. However, we will introduce a second gauge transform and second renormalization in Section~\ref{subsec:C} which motivates the change in notation.} mKdV equation (mKdV1)
\begin{align}
	\begin{cases}
		\partial_t u + \partial_x^3 u = \pm \bigg( |u|^2 -  \displaystyle \fint |u|^2 \ dx \bigg) \partial_x u, \\
		u\vert_{t=0} = u_0,
	\end{cases}
	\quad (t,x)\in\R\times\T,
	\label{renorm}
\end{align}
where $\fint f \, dx = \frac{1}{2\pi} \int_0^{2\pi} f\,dx$.
The renormalized equation \eqref{renorm} is obtained from mKdV \eqref{mkdv} through the following gauge transformation
\begin{align}
	\mathcal{G}_1 (u)(t,x) : = u(t, x \mp \mu(u(t)) t),
	\label{gauge1}
\end{align} 
where $\mu(u(t)):= \frac{1}{2\pi} \|u(t)\|_{L^2}^2$ denotes the mass, a conserved quantity of the system, i.e., $\mu(u(t)) = \mu(u_0)$. Note that mKdV1 \eqref{renorm} and the original mKdV equation \eqref{mkdv} are equivalent in $L^2(\T)$. 

Bourgain's result follows from a contraction mapping argument and it is sharp with respect to this method, since the data-to-solution map fails to be $C^3$-continuous \cite{BO97} and locally uniformly continuous in $H^s(\T)$ for~$s<\frac12$~\cite{CCT03}. We point out that Bourgain's analysis focused on real-valued inital data $u_0$, whose corresponding solution $u$ is real-valued and satisfies the following equation:
\begin{align*}
	\partial_t u + \partial^3_x u = \pm u^2 \partial_x u.
\end{align*}
The results mentioned above extend to the complex-valued setting.

In the real-valued setting, Takaoka-Tsutsumi \cite{TT04} and Nakanishi-Takaoka-Tsutsumi \cite{NTT} applied the energy method and proved the local well-posedness of mKdV in $H^s(\T)$ for $s>\frac13$. In a recent paper \cite{MPV}, Molinet-Pilod-Vento extended this result to the end-point $s=\frac13$. By exploiting the completely integrable structure of the equation, Kappeler-Topalov \cite{KT} used the inverse spectral method to show global existence and uniqueness of solutions to the real-valued defocusing mKdV (with the $+$ sign) in $H^s(\T)$, $s\geq 0$. Here, solutions are understood as the unique limit of smooth solutions and it is not required that the equation is satisfied in the sense of distributions (see \cite{KT, KOY, OH2} for further details).

Using the short-time Fourier restriction norm method, Molinet \cite{Mo} proved existence of distributional solutions for the real-valued mKdV equation in $L^2(\T)$ (without uniqueness). In the same paper, he showed that mKdV is ill-posed below $L^2(\T)$, in the sense that the data-to-solution map is discontinuous in $H^s(\T)$ for $s<0$. See also the work by Schippa \cite{Schippa}. This ill-posedness result shows the sharpness of the well-posedness theory in $L^2$-based Sobolev spaces. However, the scaling analysis suggests that the local well-posedness should hold in $H^s(\T)$ for $s>-\frac12$, as we illustrate in the following.

Consider the following symmetry of the mKdV \eqref{mkdv} equation on the real line: given a solution $u$ of mKdV \eqref{mkdv} with initial data $u_0$, then
\begin{align}
	u^\ld (t,x) = \ld u(\ld^3t,\ld  x), \quad \ld>0,
	\label{scaling}
\end{align}
is also a solution of mKdV, with rescaled initial data $\ld u_0(\ld x)$. A direct calculation shows that the homogeneous $\dot{H}^s(\R)$-norm is preserved under the scaling \eqref{scaling} when $s=-\frac12$. Thus, one would expect the local well-posedness to hold in the subcritical regime $H^s(\R)$ for $s>-\frac12$.

Although the scaling \eqref{scaling} does not hold in the periodic setting, the scaling heuristics are still relevant. The gap between the scaling prediction ($s>-\frac12$) and the ill-posedness result by Molinet ($s<0$) motivates the search for spaces with analogous scaling. One such choice are the Fourier-Lebesgue spaces $\FL^{s,p}(\T)$ defined by the norm
\begin{align}
	\|f\|_{\FL^{s,p}} = \| \jb{n}^s \ft{f}(n) \|_{\l^p_n}.\label{FL}
\end{align}
We can conduct a similar scaling analysis on the homogeneous space $\dot{\FL}^{s,p}(\R)$ defined by the norm
\begin{align*}
	\|f\|_{\dot{\FL}^{s,p}(\R)} = \big\| |\xi|^s \ft{f}(\xi) \big\|_{L^p_\xi(\R)}.
\end{align*}
It follows that the $\dot{\FL}^{s_{\text{crit}}(p),p}(\R)$-norm is invariant under the scaling \eqref{scaling}, where $s_{\text{crit}} (p) =-\frac1p$, with the convention $s_{\text{crit}}(\infty) = 0$.
Once again, transporting the scaling heuristics to the periodic setting, we say that the mKdV equation \eqref{mkdv} is scaling-critical in $\FL^{0, \infty}(\T)$. On the other hand, we have that $\dot{\FL}^{s,p}(\R)$ scales like $\dot{H}^\s(\R)$ for $\s = s + \frac1p - \frac12$.

Regarding the local-in-time analysis,  Kappeler-Molnar \cite{KM} proved the local well-posedness of the real-valued defocusing mKdV in $\FL^{s,p}(\T)$ for $s\geq 0$ and $1\leq p < \infty$, where their solutions are understood as the unique limit of classical solutions as in \cite{KT}. In view of the scaling critical regularity, this result is almost critical, in the scale of the Fourier-Lebesgue spaces. Unlike the $L^2(\T)$ solutions of \cite{KT, Mo}, the solutions in \cite{KM} are not yet known to satisfy the equation in the distributional sense. 

Now, we turn our attention to the global aspect of well-posedness. In \cite{BO2}, Bourgain proved global well-posedness of \eqref{renorm} in $H^s(\T)$ for $s\geq 1$. For the real-valued setting, Colliander-Keel-Staffilani-Takaoka-Tao \cite{CKSTT03} showed the global well-posedness in $H^s(\T)$, $s\geq \frac12$, using the $I$-method. This result was extended to $H^s(\T)$ for $s\geq 0$ for the real-valued defocusing mKdV by Kappeler-Topalov \cite{KT}, using the complete integrability of the equation. In \cite{KM}, Kappeler-Molnar proved global-in-time existence of solutions for the real-valued mKdV with small initial data in $\FL^{s,p}(\T)$, $s\geq 0$ and $1\leq p < \infty$.

In a recent paper \cite{KVZ}, Killip-Vi\c{s}an-Zhang exploited the completely integrable structure of the equation and established global-in-time a priori bounds, in the complex-valued setting. These a priori bounds, combined with the local well-posedness result in \cite{MPV}, yield the global well-posedness of the real-valued mKdV equation in $H^s(\T)$ for $s\geq \frac13$. Oh-Wang \cite{OH1} extended the result in \cite{KVZ} to the Fourier-Lebesgue setting and established global-in-time a priori bounds in $\FL^{s,p}(\T)$ for $2\leq p<\infty$ and $ 0\leq s <1 - \frac1p$ (see Proposition~\ref{prop:apriori}).

The question of well-posedness of the complex-valued mKdV equation \eqref{mkdv} has also been explored on the real line. In the Fourier-Lebesgue setting, Gr\"unrock \cite{G04} showed the local well-posedness of mKdV \eqref{mkdv} in $\FL^{s,p}(\R)$ for $s\geq 0$ and $2\leq p <4$, which was extended by Gr\"unrock-Vega \cite{GV09} to $s\geq \frac{1}{2p}$ and $1\leq p <\infty$, almost covering the full subcritical regime. These solutions can be extended in time using the a priori bounds by Oh-Wang \cite{OH1}. In a recent paper \cite{HGKV}, Harrop-Griffiths-Killip-Vi\c{s}an proved optimal local well-posedness in $H^s(\R)$ for $s>-\frac12$, exploiting the completely integrable structure of the equation.

Our goal in this paper is to study the mKdV equation in the Fourier-Lebesgue spaces, both in the real and the complex-valued settings, on the one-dimensional torus. In particular, we find that there is an additional difficulty in the low regularity complex-valued setting, which will be discussed in Section~\ref{subsec:C}. This is a phenomenon particular to the mKdV equation, since for other dispersive equations, such as the Korteweg-de Vries and the Benjamin-Ono equations, it is not necessary to distinguish between these two settings (see \cite{KPV96,CKSTT03,IonescuKenig}).

\subsection{The real-valued setting}\label{subsec:real}
We start by considering the real-valued mKdV1 equation \eqref{renorm}. Our first result is the local well-posedness of mKdV1 \eqref{renorm} below $H^\frac12(\T)$ with real-valued initial data, indeed proving that the solutions in \cite{KM}, for a restricted range of $s$ and $p$, satisfy the equation in the sense of distributions. We use the Fourier restriction norm method, with $X^{s,b}$-spaces adapted to the Fourier-Lebesgue spaces, hence not relying on the complete integrability of the equation.
\begin{theorem}\label{th:lwpreal}
	Let $(s,p)$ satisfy one of the following conditions: \textup{(i)} $\frac12 \leq s < \frac34$, ${1\leq p <\frac{4}{3-4s}}$; \textup{(ii)} $s\geq \frac34$, $1\leq p <\infty$. Then, the real-valued mKdV1 equation \eqref{renorm} is locally well-posed in $\FL^{s,p}(\T)$.
\end{theorem}
\begin{remark}\rm
	\noi (i) In view of \cite{KM}, we believe that the restriction on the range of $p$ is artificial, but we do not know how to remove it at this point.\footnote{In a recent preprint \cite{AC2}, we adapted the approach introduced by Deng-Nahmod-Yue in \cite{DNY} in the context of the periodic derivative nonlinear Schr\"odinger equation and improved Theorem~\ref{th:lwpreal} to cover the range $s\geq \frac12$ and $1\leq p < \infty$. The same improvement also holds for Theorem~\ref{th:lwp}.} See Remark~\ref{rm:nonlinear} for further discussion. Nevertheless, this range for $(s,p)$ agrees with Nguyen's local well-posedness result in the Fourier-Lebesgue spaces in \cite{N}, using the power series method in \cite{C1}. Note that the result in \cite{N} does not guarantee uniqueness of solutions.
	
	\smallskip
	\noi (ii) In a forthcoming work, we plan to combine the energy method in \cite{NTT} and the Fourier restriction norm method adapted  to the Fourier-Lebesgue spaces to improve the range of $(s,p)$ in Theorem~\ref{th:lwpreal}.
	
	\smallskip
	\noi (iii) As a consequence of the contraction mapping argument, uniqueness holds \emph{conditionally}, in the following sense: the solutions in Theorem~\ref{th:lwpreal} are unique in
	\[ C([-T,T], \FL^{s,p}(\T) ) \cap X^{s,\frac12}_{p,2}(T),  \]
	for all $0<T=T(u_0)\leq 1$ (see Definition~\ref{def:xsb}). 
	When $p=2$, it is known that \textit{unconditional} uniqueness of the mKdV equation holds in $H^s(\T)$, for $s \geq \frac13$, namely uniqueness in $C([-T,T];H^s(\T))$ (see \cite{KO, MPV}). It would also be of interest to study the uniqueness properties of solutions in the Fourier-Lebesgue spaces.
	
\end{remark}

Combining Theorem~\ref{th:lwpreal} with Oh-Wang's a priori bound, we show the global well-posedness of the real-valued mKdV1 equation \eqref{renorm} in the Fourier-Lebesgue spaces.

\begin{theorem}\label{th:gwp}
	Let $(s,p)$ satisfy one of the following conditions: \textup{(i)} $\frac12 \leq s < \frac34$, $1\leq p <\frac{4}{3-4s}$; \textup{(ii)} $s\geq \frac34$, $1\leq p <\infty$. Then, the real-valued mKdV1 equation \eqref{renorm} is globally well-posed in $\FL^{s,p}(\T)$.
\end{theorem}

\begin{remark}
	\noi (i) Theorem~\ref{th:gwp} (with restricted ranges of $s$ and $p$) extends the result of Kappeler-Molnar \cite{KM}, as it applies to the real-valued defocusing equation but also to the large data setting. Moreover, it establishes that the solutions constructed in \cite{KM} under the assumptions of Theorem~\ref{th:gwp} satisfy the equation in the sense of distributions.
	
	\smallskip
	\noi(ii) The proof of Theorem~\ref{th:gwp} is based on applying the a priori bound by Oh-Wang in \cite{OH1} to iterate the local well-posedness argument. However, the estimate requires a restriction on the regularity $s<1-\frac1p$. 
	When $\frac12\leq s<1-\frac1p$, Theorem~\ref{th:gwp} follows directly from Theorem~\ref{th:lwpreal} and the a priori bound in \cite{OH1}. When $s\geq 1 - \frac1p$, we combine the global-in-time a priori bound with a persistence of regularity argument (see Section~\ref{sec:gwp}).
	
	\smallskip
	\noi(iii) In \cite{BO94}, Bourgain proved the invariance of the Gibbs measure under the flow of the real-valued mKdV equation,
	\begin{align}
		``d\mu \  = \ Z^{-1} \exp \bigg(\mp \frac14 \int_\T u^4 \ dx -\frac12 \int_\T (\partial_x u)^2 dx\bigg) \ du ", \label{gibbs}
	\end{align}
	by establishing the local well-posedness of mKdV \eqref{mkdv} in $H^s(\T) \cap \FL^{s_1,\infty}(\T)$ for some ${s<\frac12<s_1<1}$, which includes the support of \eqref{gibbs}. The invariance of the Gibbs measure on $\FL^{s,p}(\T)$ follows from the global well-posedness result in the real-valued setting in Theorem~\ref{th:gwp}, as $\FL^{s,p}(\T)$ with $s<1-\frac1p$ includes the support of \eqref{gibbs}.
	
\end{remark}

\subsection{The complex-valued setting}\label{subsec:C}
The main focus of our paper is on the complex-valued mKdV equation. We prove that the mKdV1 equation \eqref{renorm} is ill-posed in the Fourier-Lebesgue spaces, at low regularity, and propose an alternative equation as the correct model in the low regularity setting.

We start by considering the nonlinearity $\N(u)$ of mKdV1 \eqref{renorm} on the Fourier side, omitting the time dependence,
\begin{align}
	&\ft{\N(u)} (n)  \nonumber\\
	&= \sum_{\substack{n=n_1 + n_2 + n_3\\ n_1 + n_2 \neq 0}} in_3 \ft{u}(n_1) \conj{\ft{u}}(-n_2) \ft{u}(n_3)\nonumber\\
	& = \sum_{\nbar\in\Lambda(n)} in_3 \ft{u}(n_1) \wb{\ft{u}}(-n_2) \ft{u}(n_3) - in |\ft{u}(n)|^2 \ft{u}(n) + i\bigg( \fint_\T \Im(\conj{u} \partial_x u)  dx\bigg) \ft{u}(n), \label{Cnonlinear}
\end{align}
for $\nbar = (n_1,n_2,n_3)$ and 
\begin{align*}
	\Lambda(n):= \big\{ (n_1,n_2,n_3)\in\Z^3 : \ n=n_1+n_2+n_3, \ (n_1+n_2) (n_1+n_3) (n_2+n_3) \neq 0 \big\}. 
\end{align*}

In the real-valued setting, we have $\Im(\conj{u}\partial_x u) \equiv 0$ which implies that the last term on the right-hand side of \eqref{Cnonlinear} is zero. However, in the complex-valued case, this contribution may be nonzero. 
We define the \emph{momentum} $P(u)$ as follows:
\begin{align*}
	P(u) := \fint_\T \Im( \wb{u} \partial_x u)  dx = \sum_{n\in\Z} n |\ft{u}(n)|^2,
\end{align*}
and write the nonlinearity as $\N(u) = \N^*(u) + iP(u) u$. For a solution $u\in C(\R;H^\frac12(\T))$, the momentum $P(u(t))$ is finite and conserved, but below this regularity it is not clear if it is finite let alone conserved. Consequently, a new phenomenon arises in the complex-valued setting at low regularity, as the nonlinearity \eqref{Cnonlinear} may be ill-defined. In particular, we see that the momentum is responsible for the following ill-posedness of mKdV1 \eqref{renorm} outside $H^\frac12(\T)$.

\begin{theorem}
	Let $(s,p)$ satisfy one of the following conditions: \textup{(i)} $\frac12 \leq s < \frac34$, $1\leq p <\frac{4}{3-4s}$; \textup{(ii)} $s\geq \frac34$, $1\leq p <\infty$. Suppose that $u_0\in\FL^{s,p}(\T)$ has infinite momentum in the sense that
	\begin{align*}
		|P(\P_{\leq N} u_0 )| \to \infty \text{ as } N\to \infty,
	\end{align*}
	where $\P_{\leq N}$ denotes the Dirichlet projection onto the spatial frequencies $\{|n|\leq N\}$. Then, for any $T>0$, there exists no distributional solution $u\in C([-T,T];\FL^{s,p}(\T))$ to the complex-valued mKdV1 equation \eqref{renorm} satisfying the following conditions:
	\begin{enumerate}
		\item[(i)] $u\vert_{t=0} = u_0$;
		\item[(ii)] the smooth global solutions $\{u_N\}_{N\in\NB}$ of mKdV1 \eqref{renorm}, with $u_N\vert_{t=0} = \P_{\leq N}u_0$, satisfy $u_N \to u$ in $C([-T,T];\FL^{s,p}(\T))$.
	\end{enumerate}
	\label{th:nonexistence}
\end{theorem}

\begin{remark}
	\noi(i) The second condition in Theorem~\ref{th:nonexistence} is a natural one to impose, as we would expect ``good'' solutions to have the property of being well-approximated by the smooth solutions corresponding to the truncated initial data.
	
	\smallskip 
	\noi(ii) The proof of Theorem~\ref{th:nonexistence} is based on the argument for the nonlinear Schr\"odinger equation by Guo-Oh \cite{GO}. The restricted range of $(s,p)$ follows from the need to use a local well-posedness result for a related equation \eqref{renorm2} (Theorem~\ref{th:lwp}). We do not believe this restriction to be sharp.
	
	\smallskip
	\noi(iii) An analogous non-existence result holds in the Sobolev spaces $H^s(\T)$ for $\frac13<s<\frac12$, i.e., outside $H^\frac12(\T)$. See Remark~\ref{rm:nonexist} for more details.
\end{remark}

Motivated by the ill-posedness result in Theorem~\ref{th:nonexistence}, we propose an alternative model to the complex-valued mKdV1 equation \eqref{renorm}.
Analogously to the first gauge transform $\Gg_1$ \eqref{gauge1}, which exploited conservation of mass, we introduce a second gauge transform $\Gg_2$ using the conservation of momentum to remove the singular contribution $iP(u)u$ from the nonlinearity. Given $u\in C(\R;H^\frac12(\T))$, we define the following invertible gauge transform
\begin{align*}
	\Gg_2 (u) (t,x) :=  e^{\mp i P(u) t} u(t,x).
\end{align*}

A direct computation shows that $v\in C(\R;H^\frac12(\T))$ solves mKdV1 \eqref{renorm} if and only if $u = \Gg_2(v)$ solves the \emph{second renormalized} mKdV equation (mKdV2)
\begin{align}
	\begin{cases}\displaystyle
		\partial_t u + \partial_x^3 u = \pm \Bigg( |u|^2 \partial_x u - \bigg(\fint_\T |u|^2 \, dx\bigg) \partial_x u - i\bigg( \fint_\T \Im(\wb{u} \partial_x u ) dx\bigg) u \Bigg), \\
		u\vert_{t=0} = u_0.
	\end{cases}
	\label{renorm2}
\end{align}
Focusing on the Fourier-Lebesgue spaces, for $1\leq p <\infty$ and $s> 1-\frac1p$, the gauge transform $\Gg_2$ is well-defined in $C(\R;\FL^{s,p}(\T))$ and the equations mKdV1 \eqref{renorm} and mKdV2 \eqref{renorm2} are equivalent. However, for $2\leq p<\infty$ and $\frac12 \leq s \leq 1 - \frac1p$, we have that $\FL^{s,p}(\T) \not\embeds H^\frac12(\T)$. Since outside $H^\frac12(\T)$ the momentum may be infinite, we cannot make sense of the gauge transform $\Gg_2$, and thus cannot, in general, convert solutions of mKdV2 \eqref{renorm2} into solutions of mKdV1~\eqref{renorm}.

Although any renormalization is a matter of choice, we believe that Theorem~\ref{th:nonexistence} provides evidence for our choice of $\Gg_2$. In particular, since the assumption of infinite momentum of the initial data $u_0$ can only hold if $u_0\not\in H^\frac12(\T)$, we propose mKdV2 \eqref{renorm2} as the correct model to study the complex-valued mKdV equation \eqref{mkdv} outside $H^\frac12(\T)$. To further our evidence, we establish the following the local well-posedness result for mKdV2 \eqref{renorm2} outside $H^\frac12(\T)$.
\begin{theorem}
	Let $(s,p)$ satisfy one of the following conditions: \textup{(i)} $\frac12 \leq s < \frac34$, $1\leq p <\frac{4}{3-4s}$; \textup{(ii)} $s\geq \frac34$, $1\leq p <\infty$. Then, mKdV2 \eqref{renorm2} is locally well-posed in $\FL^{s,p}(\T)$.\label{th:lwp}
\end{theorem}%

The restriction $s\geq \frac12$ is necessary if we require uniform continuity of the solution map, as shown by the following proposition.
\begin{proposition}Let $s<\frac12$ and $1\leq p < \infty$. The data-to-solution map for mKdV2 \eqref{renorm2} fails to be locally uniformly continuous in $C(\R;\FL^{s,p}(\T))$.
	\label{prop:uc}
\end{proposition}
Proposition~\ref{prop:uc} shows that we cannot use a contraction mapping argument to prove the local well-posedness of mKdV in $\FL^{s,p}(\T)$ for $s<\frac12$. Thus, Theorem~\ref{th:lwp} is sharp with respect to the method, when $s$ is concerned. However, we do not believe the restriction on $p$ to be sharp.

In order to infer on the local well-posedness of mKdV1 \eqref{renorm} in $\FL^{s,p}(\T)$ for $2\leq p < \infty$ and $\frac12 \leq s <1 - \frac1p$, we must endow the momentum with a notion of conditional convergence at low regularity. Since the momentum is not a sign definite quantity, we want to exploit the possible cancellation between positive and negative frequencies. This is achieved in the following definition, by considering symmetric truncations of the momentum.
\begin{definition}
	
	Suppose that 
	$$P(\P_{\leq N} f) \text{ converges as } N\to\infty.$$ 
	Then, we say that $f$ has finite momentum and denote the limit by $P(f)$.
	\label{def:P}
\end{definition}

The following proposition validates our notion of finite momentum as follows: consider initial data $u_0\not\in H^\frac12(\T)$ with finite momentum in the sense of Definition~\ref{def:P}; then, not only does the corresponding solution $u$ to mKdV2 \eqref{renorm2} have finite momentum but the momentum is also conserved.

\begin{proposition}
	Let $(s,p)$ satisfy one of the following conditions: \textup{(i)} $\frac12 \leq s < \frac56$, $2\leq p <\frac{6}{5-6s}$; \textup{(ii)} $s\geq \frac56$, $2\leq p <\infty$. In addition, let $u_0 \in \FL^{s,p}(\T)$ with finite momentum in the sense of Definition~\ref{def:P} and $u\in C([-T,T]; \FL^{s,p}(\T))$ the corresponding solution to mKdV2 \eqref{renorm2}. Then, we have that
	\begin{align*}
		P\big(\P_{\leq N} u(t) \big)  \to P(u_0), \quad N\to\infty,
	\end{align*}
	and we denote the limit by $P\big(u(t)\big) \equiv P(u_0)$, for each $t\in[-T,T]$.
	\label{prop:Pconservation}
\end{proposition}

In order to show Proposition~\ref{prop:Pconservation}, we follow the argument by Takaoka-Tsutsumi and Nakanishi-Takaoka-Tsutsumi (see Lemma~2.5 in \cite{TT04} and Lemma~3.1 in \cite{NTT}) and estimate the difference of momentum at time $t\in[-T,T]$ and at the initial time. Namely, we require the following energy estimate.
\begin{proposition}\label{prop:momentumEstimate}
	Let $(s,p)$ satisfy one of the following conditions: \textup{(i)} $\frac12 \leq s < \frac56$, $2\leq p <\frac{6}{5-6s}$; \textup{(ii)} $s\geq \frac56$, $1\leq p <\infty$, and $u_0\in H^\infty(\T)$. Let $u$ be a smooth solution of \eqref{renorm2} with $u\vert_{t=0}=u_0$. Then, the following estimate holds
	\begin{align*}
		\big| P(\P_{>N}u(t)) - P(\P_{>N}u(0)) \big| \lesssim \frac{1}{N^\eps} \bigg( \sup_{t'\in[0,t]}\| u(t')\|^4_{\FL^{s,p}} + \| u \|^4_{X^{s,\frac12}_{p,2}} + \| u \|^6_{X^{s,\frac12}_{p,2}}\bigg),
	\end{align*}
	for $t\in\R$, any $N\in\NB$ and $0<\eps\ll 1$ small enough, where $\P_{>N} = \Id - \P_{\leq N}$.
	
\end{proposition}

\begin{remark}\rm\
	In \cite{NTT}, the energy estimate holds in $H^\s(\T)$ for $\s>\frac13$. Taking into account that the Fourier-Lebesgue spaces $\FL^{s,p}$ scale like $H^\s$ for $\s = s +\frac1p - \frac12$, $2\leq p <\infty$, the condition $s>\frac56-\frac1p$ agrees with the restriction in \cite{NTT}. We would like to relax the regularity constraints to $s>\frac34-\frac1p$, to match the local well-posedness of mKdV2 \eqref{renorm2} (Theorem~\ref{th:lwp}). In fact, some contributions in the estimate can be controlled at this regularity. In the most difficult cases, the normal form approach assures that the estimate holds outside $H^\frac12(\T)$, but it also introduces additional resonances. Consequently, we cannot use the modulations to help estimate the multiplier, which imposes the condition $\s> \frac13$. Nevertheless, these heuristics do not imply the failure of the estimate for lower regularity, $s\leq \frac56-\frac1p$ and $\s\leq \frac13$.
\end{remark}

As a consequence of the conservation of momentum at low regularity in Proposition~\ref{prop:Pconservation}, we have the following existence result for mKdV1 \eqref{renorm}.
\begin{proposition}
	Let $(s,p)$ satisfy one of the following conditions: \textup{(i)} $\frac12 \leq s < \frac56$, $2\leq p <\frac{6}{5-6s}$; \textup{(ii)} $s\geq \frac56$, $2\leq p <\infty$, and $u_0\in\FL^{s,p}(\T)$ with finite momentum, in the sense of Definition~\ref{def:P}. Then, there exists $T>0$ and a function $u\in C([-T,T];\FL^{s,p}(\T))$ with $u\vert_{t=0} = u_0$ such that $u$ satisfies the following equation:
	\begin{align*}
		\partial_t u + \partial_x^3 u = \pm \N(u),
	\end{align*}
	in the sense of distributions, 
	where $\N(u) = \N^*(u) + i P(u) u$, where $P(u)$ is interpreted as the limit of $\{P(\P_{\leq N} u) \}_{N\in\NB}$, as $N\to\infty$.
	\label{prop:weak}
\end{proposition}

\begin{remark}\rm In order to establish the existence of solutions for the complex-valued mKdV1 equation \eqref{renorm}, we needed the following three ingredients: (i) a notion of finite momentum for the initial data, which exploited the sign indefinite nature of momentum; (ii) to show that the notion of finite momentum was strong enough to guarantee that the corresponding solutions would also have finite momentum; and (iii) that the momentum of solutions is actually conserved. Points (ii) and (iii) follow from the energy estimate in Proposition~\ref{prop:momentumEstimate}, which is responsible for the regularity constraint in Proposition~\ref{prop:weak}.
\end{remark}

We conclude this section by stating some further remarks.

\begin{remark}\rm 
	We can also consider the question of invariance of the Gibbs measure for the complex-valued mKdV equation \eqref{mkdv} and the well-posedness of this equation with randomized initial data. In particular, initial data of the following form
	\begin{align*}
		u_0(x;\omega) = \sum_{n\neq 0} \frac{g_n(\omega)}{|n|} e^{inx},
	\end{align*}
	where $\{g_n\}_{n\in\NB}$ is a family of independent standard complex-valued Gaussian random variables, i.e., real and imaginary parts are independent Gaussian random variables, with mean 0 and variance 1. It is known that $u_0 \in H^{\frac12 -}(\T) \setminus H^{\frac12}(\T)$ almost surely, therefore it is unclear if the corresponding solutions would satisfy conservation of momentum. 
	However, we can show that its momentum is finite almost surely, which gives some hope of proving the invariance of the Gibbs measure in the complex-valued setting. The momentum is given by the following quantity
	\begin{align*}
		P(u_0(\omega)) & = \sum_{n \geq 1} \frac{|g_n(\omega)|^2 - |g_{-n}(\omega)|^2}{n} = \sum_{n\neq 0} \frac{|g_n(\omega)|^2}{n}.
	\end{align*}
	Therefore, using Isserlis' Theorem we have
	\begin{align*}
		\E\big[ (P(u_0))^2 \big] & = \sum_{n,m \neq 0} \frac{\E\big[ g_n \wb{g_n} g_m \wb{g_m}\big]}{nm}
		=\sum_{n\neq 0} \frac{2 \E\big[ |g_n|^2 \big]^2}{n^2} 
		\lesssim \sum_{n\geq 1} \frac{1}{n^2} <\infty.
	\end{align*}
	Hence the momentum $P(u_0)$ is finite, almost surely.

\end{remark}

\begin{remark}\rm\label{rm:nonexist}
	The non-existence result in Theorem~\ref{th:nonexistence} is not particular to the Fourier-Lebesgue setting and can be extended to other spaces outside $H^\frac12(\T)$. In particular, the same result holds for initial data in $H^s(\T)$, $\frac13<s<\frac12$.
	By adapting the energy method in \cite{NTT} to the complex-valued setting, we can show that local well-posedness of mKdV2 \eqref{renorm2} holds in $H^s(\T)$ for $\frac13<s<\frac12$. In particular, for any sequence of smooth functions $\{u_{0n}\}_{n\in\NB}$ with $u_{0n} \to u_0$ in $H^s(\T)$, the corresponding smooth global solutions $\{u_n\}_{n\in\NB}$ converge to the solution $u$ of mKdV2 \eqref{renorm2} in $C([-T,T];H^s(\T))$. If we focus on the initial data $u_0\in H^s(\T) \setminus H^\frac12(\T)$ with infinite momentum in the following sense
	\begin{align*}
		|P\big(\P_{\leq N} u_0\big)| \to \infty \text{ as } N\to\infty, 
	\end{align*}
	we can show that there exists no distributional solution to the complex-valued mKdV1 equation \eqref{renorm} with initial data $u_0$. This follows the same argument as in the proof of Theorem~\ref{th:nonexistence}, using the local well-posedness of mKdV2 \eqref{renorm2} in $H^s(\T)$, $\frac13< s <\frac12$.
\end{remark}

\begin{remark} \rm

	The question of local well-posedness in the Fourier-Lebesgue spaces has also been pursued for the derivative nonlinear Schr\"odinger equation (DNLS):
	\begin{align*}
		i\partial_t u + \partial_x^2 u = \partial_x(|u|^2 u), \quad (t,x) \in \R\times \T.
	\end{align*}
	This study was initiated by Gr\"unrock-Herr in \cite{GH} where they established local well-posedness of DNLS in $\FL^{s,p}(\T)$ for $s\geq \frac12$ and $1 \leq p <4$, using the Fourier restriction norm method. 
	An optimal result was later established by Deng-Nahmod-Yue in \cite{DNY} through a new method inspired by probabilistic techniques. As in the case of mKdV \eqref{mkdv}, the main difficulty in the low regularity well-posedness theory for DNLS is handling the derivative loss arising from the nonlinearity. In order to overcome this problem, Herr \cite{Herr} introduced the following gauge transform
	\begin{align*}
		\Gg (u) (t,x) = e^{-i\mathcal{I}(u)(t,x)} u(t,x),
	\end{align*}
	where $\mathcal{I}(u)$ is the mean zero anti-derivative of $|u|^2 - \int_\T |u|^2 dx$. The gauge transformation $\Gg$ removes the following singular contribution in the nonlinearity
	\begin{align}
		2 \bigg(\fint \Im (u \partial_x \conj{u}) dx\bigg) u \label{P1}.
	\end{align}
	In $\FL^{\frac12,p}(\T)$, $2\leq p <\infty$, the quantity \eqref{P1} is not well-defined, but the gauge transformation $\Gg$ is continuous and invertible, which allows for the recovery of solutions of DNLS from solutions of the gauged equation. 
	In this paper, in order to overcome the derivative loss, we introduced a gauge transformation $\Gg_2$ which removes the following contribution 
	\begin{align*}
		i \bigg(\fint \Im(u \partial_x \conj{u}) dx\bigg ) u.
	\end{align*}
	However, in our case, the gauge transformation $\Gg_2$ depends explicitly on the momentum, which is not well-defined outside $H^\frac12(\T)$. Thus, we cannot freely convert solutions of mKdV2 \eqref{renorm2} to solutions of mKdV1 \eqref{renorm}, a problem which is new to the complex-valued mKdV equation, when compared to DNLS.
	This additional difficulty, not present for DNLS, lead us to the introduction of a new notion of finite momentum (Definition~\ref{def:P}) and its conservation at low regularity (Proposition~\ref{prop:Pconservation}). Only then could we prove existence of solutions of mKdV1 \eqref{renorm} in Proposition~\ref{prop:weak}.

\end{remark}

\begin{remark} \rm 
	In \cite{KishimotoTsutsumi}, Kishimoto-Tsutsumi focused on the ill-posedness of the nonlinear Schr\"odinger equation with third order dispersion and Raman scattering term:
	\begin{align*}
		\partial_t u = \al_1 \partial_x^3 u + i \al_2 \partial_x^2 u + i \gamma_1 |u|^2 u + \gamma_2 \partial_x (|u|^2 u) - i \G u \partial_x (|u|^2), \quad (t,x)\in\R\times \T,
	\end{align*}
	for $\al_j,\g_j, \G \in \R$, $j=1,2$ satisfying $\G>0$, $\al_1\neq0$ and $\frac{2\al_2}{3\al_1} \not\in \Z$. 
	Note that for $\al_2=\g_1=0$, the equation resembles mKdV \eqref{mkdv}, however, this regime is not covered in their analysis.
	The last term, the Raman scattering term, is responsible for the ill-posedness of this equation and can be rewritten as follows
	\begin{align}
		\Ft\big(u \partial_x (|u|^2) \big) (n) = & - \sum_{\substack{n=n_1+ n_2+n_3\\(n_1+n_2)(n_2+n_3)\neq0}} i(n_1 + n_2) \ft{u}(n_1) \conj{\ft{u}}(-n_2) \ft{u}(n_3)\nonumber \\
		&- in \Bigg(\sum_{n_2} |\ft{u}(n_2)|^2 \Bigg) \ft{u}(n) 
		+ \Bigg(\sum_{n_2} in_2 |\ft{u}(n_2)|^2 \Bigg) \ft{u}(n).\label{raman}
	\end{align}
	The resonance relation for this equation is 
	\begin{align*}
		\Phi(n_1,n_2,n_3) = 3\al_1(n_1+n_2) (n_2+n_3) \bigg(n_3 + n_1 + \frac{2\al_2}{3\al_1}\bigg),
	\end{align*}
	therefore, $\Phi(n_1,n_2,n_3) = 0$ if and only if $(n_1+n_2)(n_2+n_3) = 0$.  Consequently, the first term on the right-hand side of \eqref{raman} corresponds to the non-resonant contribution, analogous to $\N^*(u)$ in our case (see \eqref{Cnonlinear})
	
	Delving deeper into the Raman scattering term, note that the last two contributions on the right-hand side of \eqref{raman} can be written on the physical side as  $\big( \int_\T|u|^2 dx \big) \partial_x u$ and $ iP(u) u$, respectively. In \cite{KishimotoTsutsumi}, the term that is responsible for the ill-posedness is the last contribution in \eqref{raman}, i.e., the one that depends on the momentum. For mKdV \eqref{renorm}, both the resonant contributions in \eqref{raman} are removed by the application of the gauge transformations $\Gg_1$ and $\Gg_2$, respectively. Moreover, it is also the contribution $iP(u)u$ which is responsible for the non-existence of solutions to mKdV1 \eqref{renorm} in low regularity.
\end{remark}

This paper is organized as follows. In Section~\ref{sec:notation}, we introduce some notation and function spaces along with their relevant properties. In Section~\ref{sec:nonlinear}, we establish the main trilinear estimate. In Section~\ref{sec:nonexistence}, we show the non-existence of solutions for initial data with infinite momentum (Theorem~\ref{th:nonexistence}). The influence of the momentum on low regularity well-posedness of mKdV1 \eqref{renorm} is explored further in Section~\ref{sec:momentum}, where we establish the conservation of momentum and the existence of solutions for the complex-valued equation with initial data with finite momentum. In order to show the conservation of momentum, we prove an energy estimate for smooth solutions of mKdV2 \eqref{renorm2} in Section~\ref{sec:energy}. In Section~\ref{sec:gwp}, by establishing a modified version of the trilinear estimate, we prove the global well-posedness for the real-valued mKdV1 equation \eqref{renorm} (Theorem~\ref{th:gwp}). Lastly, in Appendix~\ref{sec:ap}, we show the failure of local uniform continuity of the solution map for $s<\frac12$, in the context of the Fourier-Lebesgue spaces (Proposition~\ref{prop:uc}).

\section{Notation, function spaces and linear estimates}\label{sec:notation}
We start by introducing some useful notation. Let  $A\lesssim B$ denote an estimate of the form $A\leq CB$ for some constant $C>0$. Similarly, $A\sim B$ will denote $A\lesssim B$ and $B\lesssim A$, while $A\ll B$ will denote $A\leq \eps B$, for some small constant $ 0<\eps< 1$.
The notations $a+$ and $a-$ represent $a+\eps$ and $a-\eps$ for arbitrarily small $\eps>0$, respectively.
Lastly, our conventions for the Fourier transform are as follows.
The Fourier transform of $u: \R\times\T \to \C$ with respect to the space variable is given by
$$\Ft_x u(t,n) =  \frac{1}{2\pi}\int_\T u(t,x) e^{-inx} \ dx.$$
The Fourier transform of $u$ with respect to the time variable is given by
$$\Ft_t u(\tau,x) =  \frac{1}{2\pi}\int_\R u(t,x) e^{- it\tau} \ dt.$$
The space-time Fourier transform is denoted by $\Ft_{t,x} = \Ft_t \Ft_x$. For simplicity, we will drop the harmless factors of $2\pi$. We will use $\ft{u}$ to denote $\Ft_x u$, $\Ft_t u$ and $\Ft_{t,x} u$, but it will become clear which one it refers to from context, namely from the use of the spatial and time Fourier variables $n$ and $\tau$, respectively.

\smallskip

Now, we focus on the relevant spaces of functions. Let $\mathcal{S} (\R\times\T)$  denote the space of functions $u:\R\times\R\to\C$, with $u\in C^\infty(\R\times\T)$ which satisfy
\begin{align*}
	u(t,x+1) = u(t,x), \quad  \sup_{(t,x) \in \R\times\T} | t^\al \partial_t^\be \partial_x^\gamma u(t,x)| < \infty, \quad \al,\be,\gamma\in\mathbb{Z}_{\geq 0}.
\end{align*}
In \cite{BO2}, Bourgain introduced the $X^{s,b}$-spaces defined by the norm
\begin{align}
	\|u\|_{X^{s,b}} = \big\| \jb{n}^s \jb{\tau - n^3}^b \ft{u}(\tau,n) \big\|_{\l^2_n L^2_\tau}.
	\label{xsb}
\end{align}
In the following, we define the $X^{s,b}$-spaces adapted to the Fourier-Lebesgue setting (see Gr\"{u}nrock-Herr \cite{GH}).

\begin{definition}\label{def:xsb}
	Let $s,b\in\R$, $1\leq p,q\leq \infty$. The space $X^{s,b}_{p,q}(\R\times\T)$, abbreviated $X^{s,b}_{p,q}$, is defined as the completion of $\mathcal{S}(\R\times\T)$ with respect to the norm
	\begin{align*}
		\|u\|_{X^{s,b}_{p,q}} = \big\| \jb{n}^s \jb{\tau - n^3}^b \ft{u}(\tau, n) \big\|_{\l^p_n L^q_\tau}.
	\end{align*}
	When $p=q=2$, the $X^{s,b}_{p,q}$-spaces defined above reduce to the standard $X^{s,b}$-spaces defined in \eqref{xsb}.
\end{definition}

Recall the following embedding. For any $1\leq p <\infty$, 
\begin{align*}
	X^{s,b}_{p,q} (\R\times \T) &\embeds C (\R; \FL^{s,p}(\T)) \quad \text{for } b>\frac{1}{q'} = 1 - \frac1q.
\end{align*}

We want to conduct a contraction mapping argument in an appropriate $X^{s,b}_{p,2}$-space. As we see in Section~\ref{sec:nonlinear}, in order to establish a trilinear estimate, we must work with $b=\frac12$. However, this space fails to be embedded into $C(\R;\FL^{s,p}(\T))$. Therefore, instead of $X^{s,\frac12}_{p,2}$, we work in $Z^{s,\frac12}_p \embeds C(\R;\FL^{s,p}(\T))$, with $Z^{s,b}_p$ defined as follows
\begin{align*}
	Z^{s,b}_p := X^{s,b}_{p,2} \cap X^{s,b-\frac12}_{p,1},
\end{align*}
with $1\leq p <\infty$, $s\in\R$ and $b>0$.

To show the local well-posedness, we will use the local-in-time versions of these spaces.
\begin{definition}
	Let $s,b\in\R$, $1\leq p,q<\infty$ and $I\subset \R$ an interval. We define the restriction space $X^{s,b}_{p,q}(I)$ of all functions $u$ which satisfy
	\begin{align*}
		\|u\|_{X^{s,b}_{p,q}(I)} := \inf \Big\{ \|v\|_{X^{s,b}_{p,q}} : \ v\in X^{s,b}_{p,q}(\R\times \T), \ v\vert_{t\in I} = u  \Big\} <\infty,
	\end{align*}
	with the infimum taken over all extensions $v$ of $u$.
	If $I = [-T,T]$, for some $0<T\leq 1$, we denote the spaces by $X^{s,b}_{p,q}(T)$. The spaces $Z^{s,b}_p(I)$ are defined analogously.
\end{definition}

Let $S(t)$ denote the linear propagator of the Airy equation, defined as follows
\begin{align*}
	\Ft_x\big(S(t)u\big)(t,n) = e^{itn^3}\ft{u}(t,n).
\end{align*}
The following linear estimates are needed to show the local well-posedness (Theorems~\ref{th:lwpreal} and \ref{th:lwp}) (see \cite{OH1, GH} for analogous proofs).
\begin{lemma}\label{lm:linear} Let $1\leq p,q <\infty$ and $s,b\in\R$. Then, the following estimates hold:
	
	\begin{align*}
		\big\| S(t) u_0 \big\|_{Z^{s,b}_{p}(T)} &\lesssim \|u_0\|_{\FL^{s,p}}, \\
		\Bigg\| \int_0^t S(t-t') F(t') \ dt' \Bigg\|_{Z^{s,b}_p (T)} &\lesssim \|F\|_{Z^{s,b-1}_{p}(T)},
	\end{align*}
	for any $0<T\leq 1$.
\end{lemma}

Lastly, we state an auxiliary result, needed for the trilinear estimate in Section~\ref{sec:nonlinear}, adapted from \cite{TAO06}.
\begin{lemma}\label{lm:time}
	Let $-\frac{1}{2} <b'\leq b <\frac{1}{2}$ and $1\leq p < \infty$. The following holds:
	\begin{equation*}
		\left\|  u \right\|_{X^{s,b'}_{p,2}(T)} \lesssim T^{b-b'} \|u\|_{X^{s,b}_{p,2}(T)},
	\end{equation*}
	for any $0<T\leq 1$.
\end{lemma}

\section{Nonlinear estimate}\label{sec:nonlinear}
In this section, we establish a fundamental trilinear estimate, required to show Theorems~\ref{th:lwpreal} and \ref{th:lwp}. 
We start by introducing the following multilinear operators, defined on the Fourier side and omitting time dependence,
\begin{align*}
	\Ft_x \big(\mathcal{NR}(u_1,u_2,u_3) \big) (n) & = \sum_{\nbar \in \Lambda(n)} in_3 \, \ft{u}_1(n_1) \ft{u}_2(n_2) \ft{u}_3 (n_3), \\
	\Ft_x \big(\mathcal{R} (u_1,u_2,u_3) \big) (n) & = i n \, \ft{u}_1(n) \conj{\ft{u}}_2 (n) \ft{u}_3(n),
\end{align*}
where $\nbar = (n_1,n_2,n_3)$ and 
$$\Lambda(n) = \{ (n_1,n_2,n_3)\in\Z^3:\ n = n_1 + n_2 + n_3, \ (n_1 + n_2) (n_1+n_3) (n_2 + n_3) \neq 0 \}.$$
Recall from \eqref{Cnonlinear} that the nonlinearity of the real-valued mKdV1 equation \eqref{renorm} and of mKdV2 \eqref{renorm2} can be written as $\mathcal{N}^*(u, u, u) = \mathcal{NR}(u, \conj{u}, u) - \mathcal{R}(u,u,u)$.
The following trilinear estimates hold.
\begin{proposition}
	Let $(s,p)$ satisfy one of the following conditions: \textup{(i)} $\frac12 \leq s < \frac34$, ${1\leq p <\frac{4}{3-4s}}$; \textup{(ii)} $s\geq \frac34$, $1\leq p <\infty$. For $u_j: \R\times \T\to \C$, $j=1,2,3$, the following estimate holds:
	\begin{align}
		\|
		\mathcal{NR}(u_1,u_2,u_3)\|_{Z^{s,-\frac12}_p(T)} &\lesssim T^\delta \prod_{j=1}^3 \|u_j\|_{X^{s,\frac12}_{p,2}(T)},\label{nonlinear_nonres}\\
		\|\mathcal{R}(u_1,u_2,u_3) \|_{Z^{s,-\frac12}_p(T)} &\lesssim T^\delta \prod_{j=1}^3 \|u_j\|_{X^{s,\frac12}_{p,2}(T)}, \label{nonlinear_res}
	\end{align} 
	for some $0<\delta \ll1$ and any $0<T\leq 1$.
	\label{prop:nonlinear}
\end{proposition}

Before proceeding to the proof of Proposition~\ref{prop:nonlinear}, recall the following well-known tools (see \cite[Lemma 4.2]{GTV} and \cite[Lemma 4.1]{MWX}, respectively).
\begin{lemma}\label{lemma:convolution}
	Let $0\leq \al\leq \be$ such that $\al + \be>1$ and $\eps>0$. Then, we have
	\begin{align*}
		\int_\R \frac{1}{\jb{x-a}^{\al} \jb{x-b}^{\be} } dx \lesssim \frac{1}{ \jb{a-b}^{\gamma} },
	\end{align*}
	where 
	\begin{align*}
		\quad \gamma =
		\begin{cases}
			\al + \be -1 , & \be<1, \\
			\al - \eps, & \be =1, \\
			\al, & \be>1.
		\end{cases}
	\end{align*}
\end{lemma}

\begin{lemma}\label{lemma:discrete_convolution}
	Let $0\leq \al,\be<1$ such that $\al + \be > 1$.
	Then, we have 
	\begin{align*}
		\sum_{\substack{n_1,n_2\in\Z \\ n_1+n_2=n}} \frac{1}{ \jb{n_1}^{\al} \jb{n_2}^{\be} } \lesssim \frac{1}{ \jb{n}^{\al+\be -1} },
	\end{align*}
	uniformly over $n\in\Z$.
\end{lemma}

\begin{proof}[Proof of Proposition~\ref{prop:nonlinear}]
	The estimates follow once we show
	\begin{align}
		\| \mathcal{NR}(\tilde{u}_1, \tilde{u}_2, \tilde{u}_3) \|_{Z^{s, -\frac12}_p} & \lesssim \max\limits_{k=1,2,3} \Bigg( \|\tilde{u}_k\|_{X^{s,\frac12}_{p,2}} \prod_{\substack{j=1,\\j\neq k}}^3 \|\tilde{u}_j\|_{X^{s,\frac12-\nu}_{p,2}}\Bigg),\nonumber\\
		\| \mathcal{R}(\tilde{u}_1, \tilde{u}_2, \tilde{u}_3) \|_{Z^{s, -\frac12}_p}  & \lesssim \prod_{j=1}^3 \|\tilde{u}_j \|_{X^{s,\frac12-\nu}_{p,2}}, \label{aux_res}
	\end{align}
	for any $\tilde{u}_j$ an extension of $u_j$ in $[-T,T]$, $j=1,2,3$, and some $\nu>0$. Then, taking an infimum over all extensions and using Lemma~\ref{lm:time} gives estimates \eqref{nonlinear_nonres} and \eqref{nonlinear_res}. For simplicity, denote the extensions $\tilde{u}_j$ by $u_j$, $j=1,2,3$, in the remaining of the proof.
	
	Let $\sigma_0 = \tau - n^3$, $\sigma_j = \tau_j - n_j^3$, $j=1,2,3$, $\mu = (\tau, n)$ and $\mu_j = (\tau_j, n_j)$, $j=1,2,3$. We start by estimating the non-resonant contribution $\mathcal{NR}$.
	
	\smallskip
	\noi\underline{\textbf{Part 1}}\\
	We first estimate the $X^{s,-\frac12}_{p,2}$-norm of $\mathcal{NR}$,
	\begin{align}
		\| \mathcal{NR}(u_1, u_2,u_3)\|_{X^{s, -\frac12}_{p,2}} & \lesssim \bigg\| \sum_{\nbar \in \Lambda(n)} \intt_{\tau=\tau_1 + \tau_2 + \tau_3}\frac{\jb{n}^s |n_3|}{\jb{\tau-n^3}^\frac12} \prod_{j=1}^3 | \ft{u}_j(\tau_j, n_j)| \bigg\|_{\l^p_n L^2_\tau}. \label{nonres-L2}
	\end{align}
	Note that $\s_0 - \s_1 - \s_2 - \s_3  = -3(n_1+n_2)(n_1+n_3)(n_2+n_3) =: \Phi(\nbar)$ which implies that 
	\begin{align}
		|\Phi(\nbar)| \lesssim \max\limits_{j=0,1,2,3} |\s_j| =: \s_{\max}. \label{resonance}
	\end{align}
	Let $|n_{\min}| \leq |n_{\med}| \leq |n_{\max}|$ denote the increasing rearrangement of the frequencies $n_1, n_2, n_3$. We distinguish the following two cases for the resonance relation $\Phi(\nbar)$:
	\begin{gather}
		|n_1|\sim|n_2|\sim|n_3| , \quad |\Phi(\nbar)| \sim |n_{\max}| \ld_1 \ld_2  \quad \text{and} \label{mod1},\\
		|\Phi(\nbar)| \sim |n_{\max}|^2 \ld, \label{mod2} 
	\end{gather}
	where $\ld,\ld_1,\ld_2 \in \{|n_1+n_2|, |n_1+n_3|, |n_2+n_3|\}$ distinct. From \eqref{resonance}, we can use the largest modulation $\s_{\max} = \max\limits_{j=0,\ldots,3} |\s_j|$ to gain powers of the maximum frequency. Thus, we will consider different cases depending on the value of $\s_{\max}$ and on which of the conditions \eqref{mod1} or \eqref{mod2} holds.
	
	\smallskip
	
	\noi\underline{\textbf{Case 1.1:} $\s_{\max} = |\s_0|$}
	
	\noi Let $f_j(\tau,n) = \jb{n}^s \jb{\tau-n^3}^{\frac12-\nu} |\ft{u}_j(\tau, n)|$ and note that $\|f_j\|_{\l^p_n L^2_\tau} = \|u_j\|_{X^{s,\frac12-\nu}_{p,2}}$, $j=1,2,3$. Then, we have 
	\begin{align*}
		\eqref{nonres-L2} & \lesssim \Bigg\| \sum_{\nbar\in\Lambda(n)} \intt_{\tau=\tau_1 + \tau_2 + \tau_3} \frac{\jb{n}^s |n_3|}{|\Phi(\nbar)|^\frac12 \prod_{j=1}^3 \jb{n_j}^s \jb{\s_j}^{\frac12-\nu}} \prod_{j=1}^3 f_j(\tau_j, n_j) \Bigg\|_{\l^p_n L^2_\tau} \\
		& \lesssim \sup_{\tau, \nbar} J_1(\tau,\nbar) \bigg\| \sum_{\nbar\in\Lambda(n)}  \frac{\jb{n}^s |n_3|}{|\Phi(\nbar)|^\frac12 \prod_{j=1}^3 \jb{n_j}^s} \bigg( \intt_{\tau=\tau_1 + \tau_2 + \tau_3}\prod_{j=1}^3 |f_j(\tau_j, n_j)|^2 \bigg)^\frac12 \Bigg\|_{\l^p_n L^2_\tau},
	\end{align*}
	using H\"older's inequality, where
	\begin{align}
		J_{1} (\tau, \nbar) := \Bigg( \intt_{\tau=\tau_1 + \tau_2 + \tau_3} \frac{1}{(\jb{\s_1} \jb{\s_2}\jb{\s_3})^{1 - 2 \nu} }  \Bigg)^\frac12. \label{J1}
	\end{align}
	Using Lemma~\ref{lemma:convolution} twice, we obtain that $J_1(\tau, \nbar) \lesssim 1$ for $0<\nu<\tfrac16$.
	Using Minkowski's and H\"older's inequalities, it follows that 
	\begin{align}
		\eqref{nonres-L2} & \lesssim \bigg\| \sum_{\nbar\in\Lambda(n)} \frac{\jb{n}^s |n_3|}{|\Phi(\nbar)|^\frac12 \prod_{j=1}^3 \jb{n_j}^s} \prod_{j=1}^3 \|f_j(n_j)\|_{L^2_\tau} \bigg\|_{\l^p_n} \nonumber\\
		& \lesssim \sup_{n} \big( J'_1(n) \big)^{\frac{1}{p'}} \prod_{j=1}^3 \|f_j\|_{\l^p_n L^2_\tau}, \label{nonlinear_main}
	\end{align}
	where
	\begin{equation}J'_1(n) = \sum_{\nbar\in\Lambda(n)} \bigg(\frac{\jb{n}^s |n_3|}{|\Phi(\nbar)|^\frac12 \prod_{j=1}^3 \jb{n_j}^s} \bigg)^{p'} . \label{J1aux}\end{equation}
	Since $\|f_j\|_{\l^p_n L^2_\tau} = \|u_j\|_{X^{s,\frac12-\nu}_{p,2}}$, it only remains to estimate $J'_1$. To that end, we must consider the two cases \eqref{mod1} and \eqref{mod2}.
	If \eqref{mod1} holds, then 
	$$\frac{\jb{n}^s |n_3| }{|\Phi(\nbar)|^\frac12 \prod_{j=1}^3 \jb{n_j}^s} \lesssim \frac{1}{ \jb{n_3}^{2s-\frac12} \jb{\ld_1}^\frac12 \jb{\ld_2}^\frac12},$$
	for distinct $\ld_1, \ld_2 \in \{|n_1+n_2| , |n_1+n_3|, |n_2+n_3|\}$. We can write $\ld_j = |n-n'_j|$, $j=1,2$, where $n'_1, n'_2$ are distinct frequencies in $n_1,n_2,n_3$. Since $\ld_1, \ld_2 \lesssim |n_3|$, we have
	\begin{align*}
		J'_1(n) &\lesssim \sum_{\nbar\in\Lambda(n)} \frac{1}{\jb{n_3}^{(2s-\frac12)p'} \jb{n-n'_1}^{\frac{p'}{2}} \jb{n-n'_2}^{\frac{p'}{2}}} \\
		& \lesssim \sum_{n'_1, n'_2} \frac{1}{\jb{n-n'_1}^{(s+\frac14)p'} \jb{n-n'_2}^{(s+\frac14)p'}} \\
		& \lesssim 1,
	\end{align*}
	for $s\geq \frac14$, $1\leq p <2$ or $s>\frac34 - \frac1p$, $2\leq p <\infty$.
	If \eqref{mod2} holds, then
	\begin{align*}
		\frac{\jb{n}^s |n_3|}{|\Phi(\nbar)|^\frac12 \prod_{j=1}^3 \jb{n_j}^s} \lesssim \frac{1}{\jb{n_{\min}}^s \jb{n_{\med}}^s \ld^\frac12},
	\end{align*}
	where $\ld \in \{|n_{\min} + n_{\med}|, |n-n_{\min}| \}$.
	If $\ld = |n_{\min} + n_{\med}|$, since $|n_{\min} | , |n_{\min} + n_{\med}| \lesssim |n_{\med}|$, we have
	\begin{align*}
		J'_1(n) & \lesssim \sum_{\nbar \in \Lambda(n)} \frac{1}{\jb{n_{\min}}^{sp'} \jb{n_{\med}}^{sp'} \jb{n_{\min} + n_{\med}}^{\frac{p'}{2}}} \\
		&\lesssim \sum_{n_{\min}, n_{\med}} \frac{1}{\jb{n_{\min}}^{(s+\frac14)p'} \jb{n_{\min} + n_{\med}}^{(s+\frac14)p'} } \lesssim 1
	\end{align*}
	given that $s\geq \frac14$, $1\leq p<2$ or $s>\frac34-\frac1p$, $2\leq p<\infty$. If $\ld = |n - n_{\min}|$, since $|n - n_{\min}|, |n_{\min}| \lesssim |n_{\med}|$, the same estimate follows from using Lemma~\ref{lemma:discrete_convolution}.

	\smallskip
	
	\noi\underline{\textbf{Case 1.2:} $\s_{\max} = |\s_j|$, $j\in\{1,2,3\}$}
	
	\noi Assume that $\s_{\max} = |\s_1|$ as a similar argument holds in the remaining cases. Let $g_1(\tau,n) = \jb{n}^s \jb{\tau-n^3}^\frac12 |\ft{u}_1(\tau,n)|$, $g_j(\tau,n) = \jb{n}^s \jb{\tau-n^3}^{\frac12 - \nu} |\ft{u}_j(\tau,n)|$, $j=2,3$, and note that $\|g_1\|_{\l^p_n L^2_\tau} = \|u_1\|_{X^{s,\frac12}_{p,2}}$ and $\|g_j\|_{\l^p_n L^2_\tau} = \| u_j \|_{X^{s,\frac12 - \nu}_{p,2}}$, $j=2,3$. Using duality, for $g_0 \in \l^{p'}_n L^2_\tau$, and H\"older's inequality, we have
	\begin{align*}
		\eqref{nonres-L2} & \lesssim \sum_n \sum_{\nbar\in\Lambda(n)} \intt_\tau \intt_{\tau = \tau_1 + \tau_2 + \tau_3} \frac{\jb{n}^s |n_3|}{|\Phi(\nbar)|^\frac12 \jb{\s_0}^{\frac12-\nu} \jb{n_1}^s \prod_{j=2}^3 \jb{n_j}^s \jb{\s_j}^{\frac12-\nu}} \\
		& \qquad \qquad \qquad \qquad \qquad \qquad \qquad \qquad \qquad  \times g_0(\tau, n) \prod_{j=1}^3 g_j(\tau_j, n_j) \\
		& \lesssim \bigg(\sup_{\tau_1,n, \nbar} J_2(\tau_1, n, \nbar) \bigg) \sum_n \sum_{\nbar\in\Lambda(n)} \frac{\jb{n}^s |n_3|}{|\Phi(\nbar)|^\frac12 \prod_{j=1}^3 \jb{n_j}^s} \|g_0(n)\|_{L^2_\tau} \prod_{j=1}^3 \|g_j(n_j)\|_{L^2_\tau},
	\end{align*} 
	where 
	\begin{align*}
		J_2(\tau_1, n, \nbar) = \bigg( \intt_{\tau_1 = \tau - \tau_2 - \tau_3} \frac{1}{ (\jb{\s_0}\jb{\s_2}\jb{\s_3})^{1 - 2\nu}} \bigg)^\frac12 \lesssim 1,
	\end{align*}
	by two applications of Lemma~\ref{lemma:convolution} with $0<\nu<\tfrac16$.
	Using H\"older's inequality, we obtain
	\begin{align*}
		\eqref{nonres-L2} & \lesssim \bigg( \sum_n \|g_0(n)\|_{L^2_\tau}^{p'} \, J'_1(n) \bigg)^{\frac{1}{p'}} \prod_{j=1}^3 \|g_j\|_{\l^p_n L^2_\tau} \lesssim \|g_0\|_{\l^{p'}_n L^2_\tau} \prod_{j=1}^3 \|g_j\|_{\l^p_n L^2_\tau}, 
	\end{align*}
	with $J'_1(n)$ defined in \eqref{J1aux}, which is uniformly bounded by following the same arguments in the previous case. This concludes the estimate for $\|\mathcal{NR}(u_1,u_2,u_3)\|_{X^{s,-\frac12}_{p,1}}$.
	
	\smallskip 
	
	\noi\textbf{\underline{Part 2}}
	
	\noi Next, we consider the $X^{s,-1}_{p,1}$-norm of $\mathcal{NR}$,
	\begin{align}
		\| \mathcal{NR}(u_1,u_2,u_3) \|_{X^{s,-1}_{p,1}} & \lesssim \bigg\| \sum_{\nbar\in\Lambda(n)} \intt_{\tau = \tau_1 + \tau_2 + \tau_3}\frac{\jb{n}^s |n_3|}{\jb{\s_0}} \prod_{j=1}^3 |\ft{u}_j(\tau_j,n_j)| \bigg\|_{\l^p_n L^1_\tau}. \label{nonres-L1}
	\end{align}
	As in Part 1, we will consider different cases depending on the value of $\s_{\max}$. If $\s_{\max} = |\s_j|$, $j\in\{1,2,3\}$, then using Cauchy-Schwarz inequality in $\tau$ gives
	\begin{align*}
		\eqref{nonres-L1} & \lesssim \bigg\| \sum_{\nbar\in\Lambda(n)} \intt_{\tau = \tau_1 + \tau_2 + \tau_3}\frac{\jb{n}^s |n_3|}{\jb{\s_0}^{\frac12-\nu}} \prod_{j=1}^3 |\ft{u}_j(\tau_j,n_j)| \bigg\|_{\l^p_n L^2_\tau}
	\end{align*}
	and the estimate follows from Case 1.2.
	Hence, we can assume that $|\s_0| \gg |\s_j|$, $j=1,2,3$, which implies that $|\s_0| \sim | \s_0 - \s_1 - \s_2 - \s_3|$. Let $h_j(\tau,n) = \jb{n}^s \jb{\tau-n^3}^{\frac12-2\nu} |\ft{u}_j(\tau,n)|$, $j=1,2,3$. Then, using H\"older's inequality with $1 = \frac{1}{q} + \frac{1}{q'}$ and $q<2$ and Minkowski's inequality, we have
	\begin{align*}
		\eqref{nonres-L1} & \lesssim \bigg\| \sum_{\nbar\in\Lambda(n)} \intt_{\tau = \tau_1 + \tau_2 + \tau_3 }\frac{\jb{n}^s |n_3|}{|\Phi(\nbar)|^\frac12 \jb{\s_0}^\frac12 \prod_{j=1}^3 \jb{n_j}^s \jb{\s_j}^{\frac12-2\nu}} \prod_{j=1}^3 h_j(\tau_j,n_j) \bigg\|_{\l^p_n L^1_\tau} \\
		& \lesssim \bigg\| \sum_{\nbar\in\Lambda(n)} \intt_{\tau = \tau_1 + \tau_2 + \tau_3 }\frac{\jb{n}^s |n_3|}{|\Phi(\nbar)|^\frac12 \prod_{j=1}^3 \jb{n_j}^s \jb{\s_j}^{\frac12-\nu}} \prod_{j=1}^3 h_j(\tau_j,n_j) \bigg\|_{\l^p_n L^q_\tau} \\
		& \lesssim \big(\sup_{\tau,\nbar} J_3(n) \big) \bigg\| \sum_{\nbar\in\Lambda(n)} \frac{\jb{n}^s |n_3|}{|\Phi(\nbar)|^\frac12 \prod_{j=1}^3 \jb{n_j}^s } \prod_{j=1}^3 \|h_j(n_j)\|_{L^q_\tau} \bigg\|_{\l^p_n},
	\end{align*}
	where 
	$$J_3(n) = \bigg( \intt_{\tau=\tau_1 +\tau_2 + \tau_3} \frac{1}{(\jb{\s_1} \jb{\s_2} \jb{\s_3})^{(\frac12 - 2\nu)q'}} \bigg)^{\frac{1}{q'}} \lesssim 1,$$
	from two applications of Lemma~\ref{lemma:convolution}, for $q$ satisfying $\frac{1}{q} > \max\big(4\nu, \frac14 + 3\nu\big)$. Using H\"older's inequality, we have
	\begin{align*}
		\eqref{nonres-L1} & \lesssim \big( \sup_{n} J'_1(n) \big)^{\frac{1}{p'}} \prod_{j=1}^3 \|h_j\|_{\l^p_n L^q_\tau},
	\end{align*}
	for $J'_1$ defined in \eqref{J1aux}. We know that $J'_1$ is uniformly bounded in $n$ from Case 1.1 and the intended estimate follows from H\"older's inequality
	\begin{align*}
		\| h_j\|_{\l^p_n L^q_\tau} = \|u_j\|_{X^{s,\frac12 - 2\nu}_{p,q}} \lesssim \|u_j\|_{X^{s, \frac12 - \nu}_{p,2}},
	\end{align*}
	given that $\frac1q < \frac12 + \nu$. Choosing $q=2-$ and $0<\nu<\tfrac18$ yields the intended result, completing the estimate of $\|\mathcal{NR}(u_1,u_2,u_3)\|_{X^{s,-1}_{p,1}}$.
	
	\smallskip 
	
	\noi\textbf{\underline{Part 3}}
	
	\noi Next, we consider the resonant part $\mathcal{R}$. Since by Cauchy-Schwarz inequality we have
	\begin{align*}
		\| \mathcal{R}(u_1,u_2,u_3)\|_{X^{s,-1}_{p,1}} \lesssim \|\mathcal{R}(u_1,u_2,u_3)\|_{X^{s,-\frac12 + \nu}_{p,2}},
	\end{align*}
	for any $\nu>0$, \eqref{aux_res} follows once we show the following estimate
	\begin{align*}
		\|\mathcal{R} (u_1,u_2,u_3)\|_{X^{s,-\frac12+\nu}_{p,2}} \lesssim \prod_{j=1}^3 \|u_j\|_{X^{s,\frac12}_{p,2}}.
	\end{align*}
	Using Cauchy-Schwarz inequality, we get
	\begin{align*}
		\|\mathcal{R}(u_1,u_2,u_3)\|_{X^{s,-\frac12+\nu}_{p,2}} & \lesssim \bigg\| \intt_{\tau= \tau_1-\tau_2 + \tau_3} \frac{\jb{n}^s|n|}{\jb{\tau - n^3}^{\frac12-\nu}} \prod_{j=1}^3 |\ft{u}_j (\tau_j,n)| \bigg\|_{\l^p_n L^2_\tau} \\
		& \lesssim \big(\sup_{\tau,n} J_4(\tau,n) \big) \bigg\| \jb{n}^s |n| \prod_{j=1}^3 \| \jb{\tau-n^3}^{\frac12 - \nu} \ft{u}_j(\tau,n)\|_{L^2_\tau} \bigg\|_{\l^p_n},
	\end{align*}
	where 
	$$J_4(\tau,n) = \bigg( \intt_{\tau=\tau_1 - \tau_2 + \tau_3} \frac{1}{(\jb{\tau-n^3} \jb{\tau_1 - n^3} \jb{\tau_2 -n^3} \jb{\tau_3 - n^3})^{1-2\nu}} \bigg)^\frac12 \lesssim 1,$$
	by two applications of Lemma~\ref{lemma:convolution}.
	Since we want $\jb{n}^s |n| \lesssim \jb{n}^{3s}$, we must impose the condition $s\geq \frac12$. Thus, using H\"older's inequality we get
	\begin{align*}
		\|\mathcal{R}(u_1,u_2,u_3) \|_{X^{s, -\frac12+\nu}_{p,2}} & \lesssim \prod_{j=1}^3 \big\| \jb{n}^s \jb{\tau-n^3}^{\frac12 - \nu} \ft{u}_j(\tau,n) \big\|_{\l^{3p}_n} \lesssim \prod_{j=1}^3 \|u_j\|_{X^{s,\frac12-\nu}_{p,2}},
	\end{align*}
	completing the estimate for the resonant contribution.
	
\end{proof}

\begin{remark}\rm\label{rm:nonlinear}
	For $s=\frac12$, Proposition~\ref{prop:nonlinear} imposes the restriction $1\leq p <4$. We do not believe this restriction to be sharp. In particular, in Case 1.1, when $\s_{\max} = |\s_0|$, the estimate holds for $1\leq p <\infty$. Note that $J_1(\tau,\nbar)$ in \eqref{J1} can be estimated as follows
	\begin{align*}
		J_1(\tau,\nbar) \lesssim \frac{1}{\jb{\tau-n^3 + \Phi(\nbar)}^{\frac12-}}.
	\end{align*}
	Then, instead of using Cauchy-Schwarz inequality on the time integrals, we can first apply H\"older's inequality on the sum $n=n_1+n_2+n_3$ and use this additional weight to help with summation.
	Unfortunately, we do not know how to extend this strategy to the cases when $\s_{\max} = |\s_j|$ for some $j\in\{1,2,3\}$.
	
\end{remark}

Theorems~\ref{th:lwpreal} and \ref{th:lwp} follow from a contraction mapping argument in $Z^{s,\frac12}_p(T)$ for some $0<T\leq 1$, by combining the linear estimates in Section~\ref{sec:notation} and the nonlinear estimates in Proposition~\ref{prop:nonlinear}.

\section{Non-existence of solutions to the complex-valued mKdV1 } \label{sec:nonexistence}
In this section, we combine the local well-posedness result for mKdV2 \eqref{renorm2} and the argument by Guo-Oh \cite{GO} to show Theorem~\ref{th:nonexistence}.

\begin{proof}[Proof of Theorem~\ref{th:nonexistence}]
	Consider $u_{0N} := \P_{\leq N} u_0$ and $\{u_N\}_{N\in\NB}$ the sequence of smooth global solutions of mKdV1 \eqref{renorm} with $u_N\vert_{t=0 } = u_{0N}$ for $N \in \NB$. Suppose that there exist $T>0$ and a solution $u\in C([-T,T];\FL^{s,p}(\T))$ to mKdV1 \eqref{renorm} such that:
	\begin{enumerate}
		\item[(i)] $u\vert_{t=0} = u_0$;
		\item[(ii)] $u_N \to u$ in $C([-T,T];\FL^{s,p}(\T))$ as $N\to \infty$.
	\end{enumerate}
	
	For the smooth solutions $u_N$, we have conservation of momentum: $P(u_N(t)) = P(u_{0N})$, $t\in[-T,T]$, $N\in\NB$. Thus, the gauge transform $\Gg_2$ is well-defined and invertible. Let $v_N:= \Gg_2(u_N)$, which is a smooth global solution of mKdV2 \eqref{renorm2} with initial data $u_{0N}$. 
	Then, by the local well-posedness of mKdV2 \eqref{renorm2}, there exists $T'=T'(\|u_0\|_{\FL^{s,p}})>0$ such that $v_N\in Z^{s,\frac12}_p(T')$, for some $T\geq T'=T'(\|u_0\|_{\FL^{s,p}})>0$\footnote{From unconditional uniqueness of mKdV2 \eqref{renorm2} at high regularity, the solutions $v_N$ coincide with the solutions constructed in Theorem~\ref{th:lwp} with initial data $u_{0N}$. Moreover, there exists $T'=T'(\|u_0\|_{\FL^{s,p}})>0$ such that $v_N \in C\big([-T',T']; \FL^{s,p}(\T)\big)$ for every $N\in\NB$. }. Now, we want to show that $\{v_N\}_{N\in \NB}$ converges in $C([-T',T'];\FL^{s,p}(\T))$. From continuous dependence of solutions of mKdV2 \eqref{renorm2} on the initial data in Theorem~\ref{th:lwp}, it follows that
	\begin{align*}
		\|v_N - v_M\|_{C_T \FL^{s,p}} & \lesssim \|v_N - v_M\|_{Z^{s,\frac12}_p(T)} \lesssim \|u_{0N} - u_{0M}\|_{\FL^{s,p}} \to 0
	\end{align*}
	as $N,M\to\infty$, since $\{u_{0N}\}_{N\in \NB}$ converges in $\FL^{s,p}(\T)$.
	Consequently, there exists $v\in C([-T',T'];\FL^{s,p}(\T))$ such that $v_N \to v$.
	
	Now, we want to exploit the rapid oscillation of the phase introduced by $\Gg_2$ to arrive at a contradiction. Let $\phi\in C_c^\infty([-T',T']\times \T)$ be a test function. Since $\FL^{s,p}(\T) \subset L^2(\T)$ for this range of $(s,p)$, $u_N \to u$ in $C([-T',T'];L^2(\T))$ and this implies
	\begin{align*}
		\jb{u_N(t, \cdot) , \phi(t, \cdot)}_{L^2_x} \to \jb{u(t, \cdot), \phi(t,\cdot)}_{L^2_x} \quad \text{ as } N\to\infty.
	\end{align*}
	Let $F(t):= \jb{u(t, \cdot), \phi(t,\cdot)}_{L^2_x}$, which is a continuous function supported on $[-T',T']$.
	Then, $F\in L^1(\R)$ and by the Riemann-Lebesgue Lemma,
	\begin{align}
		| \ft{F}(\tau)| \to 0 \text{ as } |\tau| \to \infty.
		\label{illF}
	\end{align}
	Now, we focus on the convergence of $\{v_N\}_{N\in\NB}$ in the sense of distributions. Namely, we have
	\begin{align*}
		\bigg| \int_\R\int_\T v_N \phi \ dx \ dt \bigg| & = \bigg| \int_\R\int_\T  e^{\mp i P(u_{0N})t} u_N(t,x) \phi(t,x) \ dx \ dt \bigg| \\
		& \leq |\ft{F}(\pm P(u_{0N})) | + \int_{-T'}^{T'} | \jb{u_N(t, \cdot) - u(t, \cdot), \phi(t,\cdot)}_{L^2_x} | dt \to 0 
	\end{align*}
	as $N\to\infty$. The first term converges to zero as a consequence of \eqref{illF} and the assumption that $|P(u_{0N})| \to \infty$, while the second is a consequence of $u_N\to u$ in $C([-T',T'];L^2(\T))$.
	Hence, $\{v_N\}_{N\in\NB}$ converges to zero in the sense of distributions and to $v$ in $C([-T',T'];\FL^{s,p}(\T))$. Therefore, $v\equiv 0$. However, $0 = v(0) = u_0$, which means that $P(u_0)$ must be finite, i.e., $|P(\P_{\leq N}u_0)| = |P(u_{0N})|$ converges as $N\to\infty$, which contradicts the assumption on the initial data.
\end{proof}

The non-existence of solutions for the complex-valued mKdV1 equation \eqref{renorm} for initial data with infinite momentum suggests that the mKdV2 equation \eqref{renorm2} is the correct model to study outside $H^\frac12(\T)$. In the following section, we show that imposing the conditional convergence of the momentum of the initial data (in the sense of Definition~\ref{def:P}) is sufficient for the corresponding solutions of mKdV2 \eqref{renorm2} to have finite and conserved momentum. Consequently, we can make sense of the gauge transformation $\Gg_2$ at low regularity and obtain solutions for the complex-valued mKdV1 equation \eqref{renorm}.

\section{Existence of solutions to the complex-valued mKdV1 equation with finite momentum}\label{sec:momentum}
In this section, using the energy estimate in Proposition~\ref{prop:momentumEstimate}, we show conservation of momentum at low regularity. As a consequence, we can make sense of the nonlinearity of the complex-valued mKdV1 equation \eqref{renorm} and show the existence of solutions to the complex-valued mKdV1 equation \eqref{renorm} outside $H^\frac12(\T)$.

\begin{proof}[Proof of Proposition~\ref{prop:Pconservation}]
	Let $u_{0M} = \P_{\leq M} u_0$ and $u_M$ be the corresponding smooth global solution of mKdV2 \eqref{renorm2}. Then, using Theorem~\ref{th:lwp}, there exist a time $T=T\big( \|u_0\|_{\FL^{s,p}} \big)>0$ and a solution $u \in C\big( [-T,T]; \FL^{s,p}(\T) \big)$ of mKdV2 \eqref{renorm2} such that
	\begin{equation}
		u_M \to u \quad \text{in} \quad C\big([-T,T]; \FL^{s,p}(\T)\big), \label{smooth_converge}
	\end{equation}
	as $M\to\infty$.

	In order to show convergence of $\{P(\P_{\leq N} u(t))\}_{N\in\NB}$, $t\in[-T,T]$, and its conservation, we will fix $t\in[-T,T]$ and prove the following
	\begin{align}
		P(\P_{\leq N} u(t)) &= \lim\limits_{M\to\infty} P(\P_{\leq N} u_M(t)) \label{cons1},\\
		\lim\limits_{N\to\infty} \lim\limits_{M\to\infty} P(\P_{\leq N} u_M(t)) &= \lim_{M\to\infty} P(u_M(t)) \label{cons2}.
	\end{align}
	If the two equalities hold, we have
	\begin{align*}
		\lim\limits_{N\to\infty} P(\P_{\leq N} u(t)) = \lim\limits_{M\to\infty} P(u_M(t)) = \lim\limits_{M\to\infty} P(u_{0M}) = \lim\limits_{M\to\infty} P(\P_{\leq M} u_0) = P(u_0),
	\end{align*}
	using the conservation of momentum for smooth solutions $u_M$ and the assumption of finite momentum of $u_0$, in the sense of Definition~\ref{def:P}. 
	
	We start by showing \eqref{cons1}. Note that, for each fixed $N\in\NB$,
	\begin{align*}
		& \big| P(\P_{\leq N} u)(t) - P(\P_{\leq N} u_M)(t) \big| \\
		& \quad  \leq \sum_{|n|\leq N} |n| \big|\ft{u}(t,n) - \ft{u}_M(t,n)\big| \big( |\ft{u}(t,n)| + |\ft{u}_M(t,n)|\big)\\
		& \quad \lesssim N^{\frac{p-2}{p}} \|u - u_M\|_{C_T \FL^{s,p}} \big( \|u\|_{C_T \FL^{s,p}} + \|u_M\|_{C_T \FL^{s,p}} \big),
	\end{align*}
	which implies \eqref{cons1} due to \eqref{smooth_converge}.
	
	Now, we want to show \eqref{cons2}. Since $P(\P_{\leq N} u_M(t)) = P(u_M(t)) - P(\P_{>N} u_M(t)) $, we will focus on showing that the second term goes to zero.
	Note that 
	\begin{align}
		|P(\P_{>N} u_M(t))| & \leq | P(\P_{>N} u_M(t)) - P(\P_{>N} u_{0M}) | + | P(\P_{>N} u_{0M})|. \label{cons3}
	\end{align}
	Using Proposition~\ref{prop:momentumEstimate}, for some $0<\eps\ll 1$, we have
	\begin{align*}
		| P(\P_{>N} u_M(t)) - P(\P_{>N} u_{0M}) |&\lesssim N^{-\eps} \big( \|u_{0M}\|^4_{\FL^{s,p}} + \|u_M\|^4_{C_T \FL^{s,p}} + \|u_M\|^6_{X^{s,\frac12}_{p,2}} \big) \\
		& \lesssim N^{-\eps} \big( \|u_{0}\|^4_{\FL^{s,p}} + \|u_0\|^6_{\FL^{s,p}} \big),
	\end{align*}
	which shows that $\lim\limits_{N\to\infty} \lim\limits_{M\to\infty} \big( P(\P_{>N} u_M(t)) - P(\P_{>N} u_{0M}) \big)=0$. Focusing on the last term of \eqref{cons3}, we have
	\begin{align*}
		P(\P_{>N} u_{0M}) &  = P(\P_{>N}\P_{\leq M} u_0) = P(\P_{\leq M} u_0) - P(\P_{\leq N} \P_{\leq M} u_0).
	\end{align*}
	Taking a limit as $M\to\infty$ first and then $N\to\infty$, both terms converge to $P(u_0)$ and the result follows.
	
\end{proof}

Proposition~\ref{prop:Pconservation} gives a new interpretation of finite momentum and its conservation at low regularity. Exploiting this conservation, we can make sense of the nonlinearity of the complex-valued mKdV1 equation \eqref{renorm} and show the existence of solutions, outside of $H^\frac12(\T)$.

\begin{proof}[Proof of Proposition~\ref{prop:weak}]
	
	Let $u_0\in\FL^{s,p}(\T)$ with finite momentum in the sense of Definition~\ref{def:P}. Given $N\in\NB$, let $u_{0N} = \P_{\leq N} u_0$ be smooth functions and $v_N$ be the corresponding smooth global solutions of mKdV2 \eqref{renorm2}. From Theorem~\ref{th:lwp} and a persistence of regularity argument, we can show that there exists $T=T\big(\|u_0\|_{\FL^{s,p}(\T)}\big)>0$ and a solution $v\in C\big( [-T,T]; \FL^{s,p}(\T) \big) \cap X^{s,\frac12}_{p,2}(T)$ of mKdV2 \eqref{renorm2} such that 
	\begin{equation*}
		v_N \to v \quad \text{in} \quad Z^{s,\frac12}_p(T).
	\end{equation*} 
	Since $v_N$ are smooth solutions, conservation of momentum holds and $P(v_N(t)) = P(u_{0N})$ for all $t\in\R$. 
	Let $u_N := \Gg_2^{-1}(v_N) = e^{\pm i P(u_{0N}) t} v_N$, which is a smooth global solution of mKdV1 \eqref{renorm} with initial data $u_{0N}$, $N\in\NB$. We want to show that the sequence $\{u_N\}_{N\in\NB}$ converges to $u:=e^{\pm i P(u_0)t} v$ in $Z^{s,\frac12}_p (T)$. The limit $u$ will be our candidate solution in $C([-T,T];\FL^{s,p}(\T))$.
	First, we have that
	\begin{align*}
		\|u_N - u\|_{C_T\FL^{s,p}} 
		& \leq T\big| P(u_{0N}) - P(u_0) \big| \|v\|_{C_T\FL^{s,p}} + \|v_n - v \|_{C_T\FL^{s,p}} \to 0,
	\end{align*}
	using the assumption on the momentum and the convergence of $\{v_N\}_{N\in\NB}$.
	Moreover, $u\in Z^{s,\frac12}_p (T)$, since
	\begin{align*}
		\| u\|_{Z^{s,\frac12}_p(T)} 
		& \lesssim \jb{P(u_0)}^\frac12 \|v\|_{X^{s,\frac12}_{p,2}(T)} + \|v\|_{X^{s,0}_{p,1}(T)} <\infty.
	\end{align*}
	If we show that the sequence $\{u_N\}_{N\in\NB}$ is Cauchy in $Z^{s,\frac12}_p(T_*)$ for some $0<T_*\leq T$, the convergence to $u$ in this space will follow. For $N,M\in\NB$, $u_N$ and $u_M$ are smooth solutions of mKdV1 \eqref{renorm}, thus using Lemma~\ref{lm:linear} and Proposition~\ref{prop:nonlinear}, we have
	\begin{multline*}
		\|u_N - u_M\|_{Z^{s,\frac12}_p(T)} \leq C_1 \|u_{0N} - u_{0M} \|_{\FL^{s,p}} + C_2 T^\delta |P(u_{0N}) - P(u_{0M})| \|u_{N}\|_{Z^{s,\frac12}_p(T)} \\
		+ C_3 T^\delta \Bigg( \Big(\|u_N\|_{Z^{s,\frac12}_p(T)} + \|u_M\|_{Z^{s,\frac12}_p(T)} \Big)^2  + |P(u_{0M})| \Bigg)\|u_N - u_M\|_{Z^{s,\frac12}_p(T)},
	\end{multline*}
	for some constants $C_1,C_2,C_3>0$.
	By the definition of $u_N$ and the continuous dependence on the initial data for mKdV2 \eqref{renorm2}, for large enough $N$, we have $\|u_N\|_{Z^{s,\frac12}_p(T)} \leq C( \|u_0\|_{\FL^{s,p}} + 1 )$, for some $C>0$. Analogously, for large enough $N$, $|P(u_{0N})| \leq |P(u_0)| + 1 $. Consequently, 
	\begin{align*}
		& \|u_N - u_M\|_{Z^{s,\frac12}_p(T)}\\
		& \quad \leq C_1 \|u_{0N} - u_{0M} \|_{\FL^{s,p}} + C C_2 T^\delta (\|u_0\|_{\FL^{s,p}} + 1) |P(u_{0N}) -P(u_{0M})| \\
		& \quad + C_3T^\delta \Big( 4 C^2 \big( \|u_0\|_{\FL^{s,p}} + 1 \big)^2 + \big(|P(u_0)|+1) \Big) \|u_N - u_M \|_{Z^{s,\frac12}_p(T)},
	\end{align*}
	for $N,M$ large enough.
	Choosing $0<T_0\leq T$ such that
	\begin{align*}
		C_3 T_0^\delta \Big( 4 C^2 \big( \|u_0\|_{\FL^{s,p}} + 1 \big)^2 + \big(|P(u_0)|+1) \Big) < \frac12,
	\end{align*}
	it follows that 
	\begin{equation}
		\begin{aligned}
			\|u_N - u_M\|_{Z^{s,\frac12}_p(T_0)} &\leq 2 C_1 \|u_{0N} - u_{0M} \|_{\FL^{s,p}} \\
			& \quad + 2C C_2 T_0^\delta (\|u_0\|_{\FL^{s,p}} + 1) |P(u_{0N}) -P(u_{0M})|. 
		\end{aligned}\label{aux2_exist}
	\end{equation}
	By iterating this approach, we can cover the whole interval $[-T,T]$ and the estimate \eqref{aux2_exist} holds with $T$ instead of $T_0$.
	Thus, $\{u_N\}_{N\in\NB}$ is a Cauchy sequence in $Z^{s,\frac12}_p (T)$ and $u_N \to u$ in $Z^{s,\frac12}_p(T)$. 
	
	Now, we want to show that $u$ satisfies mKdV1 \eqref{renorm} in the sense of distributions, with the nonlinearity interpreted as 
	\begin{align*}
		\mathbf{N}(u):= \N^*(u) + iP(u_0)u.
	\end{align*}
	Considering the linear part and any test function $\phi \in C^\infty_c([-T,T] \times \T)$, it follows that
	\begin{align*}
		\big|\jb{u-u_N, (\partial_t + \partial_x^3) \phi}_{t,x} \big| 
		& \lesssim_{\phi} \|u-u_N\|_{X^{s, \frac12}_{p,2}(T)} \to 0,
	\end{align*}
	as $N\to\infty$, which implies that $(\partial_t + \partial_x^3)u_N \to (\partial_t + \partial_x^3)u$ in the sense of distributions.	
	
	For the nonlinearity, using the fact that $\N(u_N) = \N^*(u_N) + iP(u_N)u_N$, it follows that
	\begin{multline*}
		\big| \jb{\N(u_N) - \mathbf{N}(u), \phi}_{t,x}\big| 
		\lesssim_{\phi} \| \N^*(u_N) - \N^*(u) \|_{X^{s,-\frac12}_{p,2}(T)} \\
		+ |P(u_{0N}) - P(u_{0})| \|u_N\|_{X^{s,\frac12}_{p,2}(T)}
		+ |P(u_{0})| \|u_N - u\|_{X^{s,\frac12}_{p,2}(T)} .	
	\end{multline*}
	Using the convergence of momentum $P(u_{0N})  \to P(u_0)$ and of $\{u_N\}_{N\in\NB}$, it suffices to estimate the first term on the right-hand side. 
	We can write $\N^*(u_N) - \N^*(u) = \N^*(u_N-u, u_N, u_N) + \N^*(u, u_N-u, u_N) + \N^*(u, u, u_N-u)$ and using the nonlinear estimate in Proposition~\ref{prop:nonlinear}, we have that
	\begin{align*}
		\| \N^*(u_N) - \N^*(u) \|_{X^{s,-\frac12}_{p,2}(T) } 
		& \lesssim \|u_N - u\|_{X^{s,\frac12}_{p,2}(T)} \bigg( \|u_N\|_{X^{s,\frac12}_{p,2}(T)} + \|u\|_{X^{s,\frac12}_{p,2}(T)} \bigg)^2,
	\end{align*}
	and the convergence follows from that of $\{u_N\}_{N\in\NB}$.
	The limit $u$ satisfies the following equation
	\begin{align*}
		\partial_t u + \partial_x^3 u = \pm \big( \N^*(u) + iP(u_0) u \big),
	\end{align*}
	in the sense of distributions, where $P(u_0)$ is interpreted in the sense of Definition~\ref{def:P}.
	
\end{proof}

\section{Momentum estimate}\label{sec:energy}
In this section, we establish an energy estimate on smooth solutions of the mKdV2 equation \eqref{renorm2}, namely we prove Proposition~\ref{prop:momentumEstimate}. This proof follows the argument by Nakanishi-Takaoka-Tsutsumi \cite{NTT} and is essential in showing the conservation of momentum at low regularity.

We start by recalling some embeddings used in the proof. From \cite{BO2}, we have the following $L^6$-Strichartz estimates
\begin{align}
	X^{0+,\frac12+}_{2,2} \subset L^6_{t,x} \label{L6}.
\end{align}
Interpolating \eqref{L6} with the Sobolev inequality $X^{\frac13+, \frac13+}_{2,2} \subset L^6_{t,x}$, we have the following
\begin{align}
	X^{0+, \frac12-}_{2,2} \subset L^6_{t,x}. \label{L6new}
\end{align}
We will also need the fact that multiplication by a sharp cut-off is a bounded operation in $X^{s,b}_{2,2}$ (see \cite{DD}, for example).
\begin{lemma}
	Let $s\geq 0$, $0\leq b <\frac12$ and fix $T>0$. Then, the following estimate holds
	\begin{align*}
		\|\chi_{[0,T]}(t) u \|_{X^{s,b}_{2,2}} \lesssim \|u\|_{X^{s,b}_{2,2}}.
	\end{align*}
	\label{lm:cutoff}
\end{lemma}

\begin{proof}[Proof of Proposition~\ref{prop:momentumEstimate}]
	Using the Fundamental Theorem of Calculus and the equation \eqref{renorm2} on the Fourier side, we have the following
	\begin{align*}
		& \big|P\big(\P_{>N}u(t)\big) - P\big(\P_{>N}u(0)\big)\big| \\
		& \quad = \Bigg| \sum_{|n|>N} n \big( |\ft{u}(t,n)|^2 - |\ft{u}(0,n)|^2 \big) \Bigg| \\
		& \quad  = \Bigg| 2 \sum_{|n|>N} n  \Re \int_0^t (\partial_t  \ft{u}(t',n) ) \conj{\ft{u}} (t',n) \ dt' \Bigg| \\
		&\quad  = \Bigg| 2 \Im \int_0^t \sum_{|n|>N} \sum_{\nbar\in\Lambda(n)} nn_3  \ft{u}(t',n_1) \conj{\ft{u}}(t',-n_2) \ft{u}(t',n_3) \conj{\ft{u}}(t',n) \ dt' \Bigg|.
	\end{align*}
	Let $|n_{\min}| \leq |n_{\med}| \leq |n_{\max}|$ denote the ordered rearrangement of $n_1,n_2,n_3$. We will consider the following 6 cases depending on the relative size of the frequencies:
	\begin{itemize}
		\item Case 1: $ |n_{\max}| \gg | n_{\med}| \gtrsim |n_3|$ or $|n_{\max}| \sim |n_{\med} | \gg |n_3|$
		\item Case 2: $ |n_{\max}| \gg |n_3| \gg |n_{\min}| $
		\item Case 3:  $|n_3| \sim |n_{\med}| \gg |n_{\min}|$
		\item Case 4:  $|n_3| \gg |n_{\med}| \geq |n_{\min}| \ges |n_3|^\frac12$ or $|n_3| \gg |n_{\med}| \ges |n_3|^\frac12 \gg |n_{\min}|$
		\item Case 5: $|n_1| \sim |n_2| \sim |n_3|$
		\item Case 6: $ |n_3|^\frac12 \gg |n_1|, |n_2|$
	\end{itemize}
	In Cases 1--4, the difference can be estimated directly, while in Cases 5 and 6 we will require the normal form approach.
	
	\medskip
	\noi\underline{\textbf{Part 1}}\\
	We start by focusing on Cases 1--4. 
	Let $\s_j:= \tau_j - n_j^3$, $j=1,2,3$, and $\s_0:= \tau - n^3$ denote the modulations. The following relation holds
	\begin{equation*}
		-\s_0 + \s_1 + \s_2 + \s_3 = n^3 - n_1^3 - n_2^3 - n_3^3 =  \Phi(\nbar).
	\end{equation*}
	In Cases 1--4, the resonance relation $\Phi(\nbar)$ satisfies the following
	\begin{equation*}
		|n_{\max}|^2 \ld \sim |\Phi(\nbar)| \lesssim \s_{\max} := \max_{j=0,\ldots,3}|\s_j|, 
	\end{equation*}
	where $\lambda \in \{ |n_1+n_2|, |n_1 + n_3| , |n_2+n_3| \}$. Let $\mu_j = (\tau_j, n_j)$, $j=1,\ldots,3$, $\mu = (\tau,n)$ and assume that $\s_{\max} = |\s_0|$, as the remaining cases can be handled analogously. In order to extend the integral from $[0,t]$ to the whole real line, we must associate the time-cutoff with one of the factors, for example $\ft{u}(t,n_1)$. Using Parseval's identity, we have that
	\begin{multline}
		|P(\P_{>N}u(t)) - P(\P_{>N}u(0))|
		= \Bigg| 2 \Im \intt_{\tau = \tau_1 + \tau_2 + \tau_3 }\sum_{|n|>N} \sum_{\nbar \in \Lambda(n)} nn_3 \\
		\times  \Ft_{t,x} \big( \chi_{[0,t]} u\big)(\mu_1) \conj{\ft{u}}(-\mu_2) \ft{u}(\mu_3) \conj{\ft{u}}(\mu) \, d\tau_1  \, d\tau_2 \, d\tau_3 \Bigg|. \label{energy1}
	\end{multline}
	
	We focus on showing the estimate for Case 4, as the remaining cases follow similar arguments.
	In Case 4 (i), $|\Phi(\nbar)| \ges |n_3|^2$ and we can estimate the multiplier as 
	$$\frac{|n n_3|}{|\Phi(\nbar)|^\frac12} \lesssim |n_3| \lesssim N^{0-} (\jb{n}\jb{n_1}\jb{n_2}\jb{n_3})^{\frac13+}.$$
	In Case 4 (ii), $|\Phi(\nbar)| \sim |n_3|^2 |n_{\med}| \ges |n_3|^{\frac52}$. Thus, 
	$$\frac{|nn_3|}{|\Phi(\nbar)|^\frac12} \lesssim |n_3|^{\frac34} \lesssim N^{0-} (\jb{n} \jb{n_{\med}} \jb{n_3})^{\frac{3}{10}+}.$$
	Consider the following notation
	\begin{align*}
		\ft{f}_1(\tau,n) &= \jb{n}^{\frac13+} \big| \Ft_{t,x}(\chi_{[0,t] } u) (\tau,n) \big|,\\ \ft{f}_2(\tau,n) &= \jb{n}^{\frac13+} | \ft{u} (\tau,n)|,\\
		\ft{f}_3(\tau,n) &= \jb{n}^{\frac13+} \jb{\tau-n^3}^\frac12 |\ft{u}(\tau,n)|.
	\end{align*}
	For both cases, using H\"older's inequality, $L^6$-Strichartz \eqref{L6new} and Lemma~\ref{lm:cutoff}, we have
	\begin{align*}
		\eqref{energy1} 
		& \lesssim \frac{1}{N^{0+}} \|  f_1 f_2^2 f_3 \|_{L^1_{t,x}}\\
		& \lesssim \frac{1}{N^{0+}}  \|f_1\|_{L^6_{t,x}} \|f_2\|_{L^6_{t,x}}^2 \| f_3\|_{L^2_{t,x}}\\
		& \lesssim \frac{1}{N^{0+}} \|\chi_{[0,t]}u\|_{X^{\frac14+,\frac12-}_{2,2}} \|u\|_{X^{\frac14+, \frac12-}_{2,2}}^2 \|u\|_{X^{\frac14+, \frac12}_{2,2}}\\
		&\lesssim \frac{1}{N^{0+}} \|u\|_{X^{s,\frac12}_{p,2}}^4,
	\end{align*}
	for $1\leq p <\infty$ and $s>\max\big(\frac13, \frac56-\frac1p\big)$.

	\medskip
	\noi\underline{\textbf{Part 2}}\\
	We now focus on Cases 5--6. 
	Since $P(\P_{>N}u(t) ) = P(\P_{>N}v(t))$, where $v(t) = S(-t) u(t)$ stands for the interaction representation, the difference of momenta can be written as follows, in terms of $v$,
	\begin{multline*}
		P\big(\P_{>N}v(t)\big) - P\big(\P_{>N}(v(0)\big)  \\
		= - 2 \Im \int_0^t \sum_{|n|>N} \sum_{\nbar\in\Lambda(n)}\!\!\! nn_3 e^{-it'\Phi(\nbar)} \ft{v}(n_1) \conj{\ft{v}}(-n_2) \ft{v}(n_3) \conj{\ft{v}}(n) \, dt'.
	\end{multline*}
	Using integration by parts, we obtain
	\begin{align*}
		& \Im \int_0^t \sum_{|n|>N} \sum_{\nbar\in\Lambda(n)} nn_3 \frac{d}{dt}\bigg( \frac{e^{-it'\Phi(\nbar)}}{-i\Phi(\nbar)} \bigg) \ft{v}(n_1) \conj{\ft{v}}(-n_2) \ft{v}(n_3) \conj{\ft{v}}(n) \, dt' \\
		& \phantom{XXXXX} =  -\Re \sum_{|n|>N} \sum_{\nbar\in\Lambda(n)} \frac{nn_3}{\Phi(\nbar)} \bigg( e^{-it\Phi(\nbar)} \ft{v}(t,n_1) \conj{\ft{v}}(t,-n_2) \ft{v}(t,n_3) \conj{\ft{v}}(t,n) \\
		& \phantom{XXXXXXXXXXXXXXXXXXX}- \ft{v}(0,n_1) \conj{\ft{v}}(0,-n_2) \ft{v}(0,n_3) \conj{\ft{v}}(0,n) \bigg) \\ 
		&\phantom{XXXXXxx} + \Re \int_0^t \sum_{|n|>N} \sum_{\nbar\in\Lambda(n)} \frac{nn_3}{\Phi(\nbar)} e^{-it' \Phi(\nbar)} \partial_t\big( \ft{v}(n_1) \conj{\ft{v}}(-n_2) \ft{v}(n_3) \conj{\ft{v}}(n) \big) \, dt'.
	\end{align*}
	In order to estimate the last term on the right-hand side, we will use the equation for $v$ again, substituting the time derivative by the corresponding resonant and non-resonant nonlinear terms. Therefore, writing the terms depending on $u$, we are interested in estimating the following quantities
	\begin{align*}
		\mathcal{B}(t) & = \Re \sum_{|n|>N} \sum_{\nbar\in\Lambda(n)} \frac{nn_3}{\Phi(\nbar)} \ft{u}(t,n_1) \conj{\ft{u}}(t,-n_2) \ft{u}(t,n_3) \conj{\ft{u}}(t,n), \\
		\mathcal{R}_0 & = \Im \int_0^t \sum_{|n|>N} \sum_{ \nbar\in\Lambda(n)} \frac{n^2 n_3}{\Phi(\nbar)} \ft{u}(n_1)\conj{\ft{u}}(-n_2) \ft{u}(n_3) \conj{\ft{u}}(n) |\ft{u}(n)|^2 \, dt', \\
		\mathcal{R}_1 & = \Im \int_0^t \sum_{|n|>N} \sum_{\nbar\in \Lambda(n)} \frac{nn_1 n_3}{\Phi(\nbar)} \ft{u}(n_1) |\ft{u}(n_1)|^2 \conj{\ft{u}}(-n_2) \ft{u}(n_3) \conj{\ft{u}}(n) \, dt' ,\\
		\mathcal{R}_2 & = \Im \int_0^t \sum_{|n|>N} \sum_{\nbar\in \Lambda(n)} \frac{nn_2 n_3}{\Phi(\nbar)} \ft{u}(n_1)  \conj{\ft{u}}(-n_2) |\ft{u}(-n_2)|^2 \ft{u}(n_3) \conj{\ft{u}}(n) \, dt' ,\\
		\mathcal{R}_3 & = \Im \int_0^t \sum_{|n|>N} \sum_{\nbar\in\Lambda(n)} \frac{n n_3^2}{\Phi(\nbar)} \ft{u}(n_1)  \conj{\ft{u}}(-n_2) \ft{u}(n_3) |\ft{u}(n_3)|^2\conj{\ft{u}}(n) \, dt' ,\\
		\mathcal{NR}_0 & =\Im\int_0^t \sum_{|n|>N} \sum_{\substack{\nbar\in\Lambda(n),\\\mbar\in \Lambda(-n)}} \frac{nn_3m_3 }{\Phi(\nbar)} \ft{u}(n_1) \conj{\ft{u}}(-n_2) \ft{u}(n_3) \conj{\ft{u}}(-m_1) {\ft{u}}(m_2) \conj{\ft{u}}(-m_3) \, dt',\\
		\mathcal{NR}_1 & = \Im\int_0^t \sum_{|n|>N} \sum_{\substack{\nbar\in\Lambda(n),\\ \mbar \in \Lambda(n_1)}} \frac{nn_3m_3}{\Phi(\nbar)} \conj{\ft{u}}(-n_2) \ft{u}(n_3) \conj{\ft{u}}(n) \ft{u}(m_1) \conj{\ft{u}}(-m_2) \ft{u}(m_3) \, dt',\\
		\mathcal{NR}_2 & = \Im\int_0^t \sum_{|n|>N} \sum_{\substack{\nbar\in\Lambda(n),\\ \mbar\in\Lambda(n_2)}} \frac{nn_3m_3}{\Phi(\nbar)} \ft{u}(n_1) \ft{u}(n_3) \conj{\ft{u}}(n) \conj{\ft{u}}(-m_1) {\ft{u}}(m_2) \conj{\ft{u}}(-m_3)  \, dt',\\
		\mathcal{NR}_3 & = \Im\int_0^t \sum_{|n|>N} \sum_{\substack{\nbar\in\Lambda(n),\\ \mbar \in \Lambda(n_3)}} \frac{nn_3m_3}{\Phi(\nbar)} \ft{u}(n_1) \conj{\ft{u}}(-n_2) \conj{\ft{u}}(n) \ft{u}(m_1) \conj{\ft{u}}(-m_2) \ft{u}(m_3)  \, dt',
	\end{align*}
	where $\mbar = (m_1,m_2,m_3)$.
	
	\medskip
	\noi$\bullet$ \textbf{Estimate for $\mathcal{B}(t)$}
	
	\smallskip
	\noi\textbf{\underline{Case 5: $|n_1| \sim |n_2| \sim |n_3|$}}\\
	Note that $|\Phi(\nbar)| \sim |n_3| \ld_1 \ld_2$, where $\ld_1,\ld_2 \in \{ |n_1+n_2|, |n_1+n_3|, |n_2+n_3|\}$, $\ld_1\neq\ld_2$. Assume that $\ld_1 = |n_1+n_3|, \ld_2 =|n_2+n_3|$. We will omit the estimate for the remaining choices of $\ld_1, \ld_2$, as it follows an analogous approach. Therefore, we have that
	\begin{align*}
		\frac{|nn_3|}{|\Phi(\nbar)| (\jb{n}\jb{n_1}\jb{n_2}\jb{n_3})^{\frac14+}} \lesssim \frac{1}{N^{0+}\jb{n_1+n_3} \jb{n_2+n_3}}.
	\end{align*}
	Hence, with $g(t,n) = \jb{n}^s |\ft{u}(t,n)|$, using H\"older's inequality and the fact that $|n|\lesssim |n_j|$, $j=1,2,3$, it follows that
	\begin{align*}
		|\mathcal{B}(t)| & \lesssim \frac{1}{N^{0+}} \sum_{n, n_1, n_2} \frac{g(t,n_1) g(t,-n_2) g(t,n - n_1 - n_2) g(t,n)}{\jb{n-n_2} \jb{n-n_1} \jb{n}^{4(s-\frac14)}}  \\
		& \lesssim \frac{1}{N^{0+}} \Bigg( \sum_{n,n_1,n_2} \frac{g(t,n)^{p'} }{\jb{n-n_2}^{p'} \jb{n-n_1}^{p'} \jb{n}^{4(s-\frac14)p'}} \Bigg)^{\frac{1}{p'}} \|g(t)\|_{\l^p}^3 \\
		& \lesssim  \frac{1}{N^{0+}} \|u(t)\|^4_{\FL^{s,p}},
	\end{align*}
	for $1\leq p <\infty$ and $s>\max\big(\frac12-\frac{1}{2p}, \frac14\big)$.
	
	\smallskip
	\noi\textbf{\underline{Case 6: $|n_3|^\frac12 \gg |n_{\med}| \ges |n_{\min}|$ }}\\
	Assume that $n_{\med}=n_2, n_{\min}=n_1$, as the estimate is analogous otherwise. Since $|\Phi(\nbar)| \sim |n_3|^2 |n_1 + n_2|$, we control the multiplier as follows
	\begin{align*}
		\frac{|nn_3| }{|\Phi(\nbar)|} \lesssim \frac{1}{\jb{n_1+n_2}}.
	\end{align*}
	Using H\"older's inequality, we have
	\begin{align*}
		|\mathcal{B}(t)| & \lesssim \frac{1}{N^{0+}}  \Bigg(\sum_{n_1,n_2,n_3} \frac{g(t,n_1)^{p'}}{\jb{n_1 + n_2}^{p'} \jb{n_1}^{2sp'} \jb{n_3}^{sp'} \jb{n_1+n_2+n_3}^{sp'-}}  \Bigg)^{\frac{1}{p'}} \|g(t)\|^3_{\l^p_n}\\
		&  \lesssim \frac{1}{N^{0+}}   \Bigg(\sum_{n_1} \frac{g(t,n_1)^{p'}}{ \jb{n_1}^{2sp'} } \Bigg)^{\frac{1}{p'}} \|g(t)\|^3_{\l^p_n}\\
		&\lesssim \frac{1}{N^{0+}} \|u(t)\|_{\FL^{s,p}}^4 ,
	\end{align*}
	for $1\leq p \leq2$, $s>0$ or $2<p<\infty$, $s>\frac12-\frac{1}{2p}$.
	
	\medskip
	\noi$\bullet$ \textbf{Estimate for $\mathcal{R}_j$, $j=0,1,2,3$}\\
	We will focus on estimating $\mathcal{R}_0$. The estimate for the remaining contributions follows by a similar approach. Let the following notation denote the modulations of the 6 factors
	\begin{align*}
		\s_j &= \tau_j-n_j^3 , \ j=1,2,3, \\
		\quad \s_4 &= \tau_4 + n^3, \quad \s_5  = \tau_5 - n^3, \quad
		\s_6  = \tau_6 + n^3,
	\end{align*}
	which implies that $ |\Phi(\nbar)| = |\s_1 + \ldots + \s_6| \lesssim \max_{j=1,\ldots,6} |\s_j|$. 
	Assume that $|\s_1|$ is the largest modulation. Then, we can associate the time cut-off with the second factor. If another $|\s_j|$ is the largest modulation, we can associate the cut-off with the first factor and the estimate follows analogously. Note that we can rewrite $\mathcal{R}_0$ as follows
	\begin{align*}
		\mathcal{R}_0  &= \Im \sum_{|n|>N} \sum_{ \nbar\in\Lambda(n)} \frac{n^2 n_3}{\Phi(\nbar)}  \Ft_t\bigg( \ft{u}(n_1) (\chi_{[0,t]} \conj{ \ft{u}})(-n_2)  \ft{u}(n_3) \conj{\ft{u}}(n) \ft{u}(n) \conj{\ft{u}}( n) \bigg)(0)\\
		&=\Im\intt_{\tau_1 + \ldots + \tau_6 = 0} \sum_{|n|>N} \sum_{ \nbar\in\Lambda(n)} \frac{n^2 n_3}{\Phi(\nbar)} \ft{u}(\tau_1,n_1)\conj{ \Ft\big(\chi_{[0,t]} u\big)}(-\tau_2,-n_2)\\  
		&\qquad \qquad \qquad \qquad \qquad \qquad \times \ft{u}(\tau_3,n_3) \conj{\ft{u}}(-\tau_4,n) \ft{u}(\tau_5,n) \conj{\ft{u}}(-\tau_6,n) \ d \tau_1 \cdots d\tau_5.
	\end{align*}
	Using the following notation 
	\begin{align*}
		g_1(\tau,n)&= \jb{n}^s \jb{\tau - n^3}^{\frac12} |\ft{u}(\tau,n)| ,\\
		g_2(\tau,n)&= \jb{n}^s \jb{\tau - n^3}^{\frac12-} \big| \Ft\big(\chi_{[0,t]} u\big)(\tau, n) \big|,
	\end{align*}
	apply Cauchy-Schwarz inequality to obtain the following estimate
	\begin{multline*}
		|\mathcal{R}_0|  \lesssim \frac{1}{N^{0+}} \sum_{|n|>N} \sum_{\nbar\in\Lambda(n)} \frac{|n|^{2+} |n_3 |}{|\Phi(\nbar)|^\frac32 (\jb{n_1} \jb{n_2} \jb{n_3})^s \jb{n}^{3s}}  \|g_1(-n_2)\|_{L^2_\tau}  \\ \times  \|g_1(n_3)\|_{L^2_\tau} \|g_1(n)\|_{L^2_\tau}^3  \Bigg( \int \frac{|g_2(\tau_1, n_1)|^2 }{\jb{\s_2}^{1-} \jb{\s_3}\cdots \jb{\s_6}} \ d\tau_1 \ldots d\tau_5 \Bigg)^\frac12  .
	\end{multline*}
	By applying Lemma~\ref{lemma:convolution} we estimate the last factor on the right-hand side by $\|g_1 (n_1)\|_{L^2_\tau}$ and the problem reduces to showing
	\begin{multline}
		\sum_{|n|>N} \sum_{\nbar\in\Lambda(n)} \frac{|n|^{2+} |n_3 |}{|\Phi(\nbar)|^\frac32 (\jb{n_1} \jb{n_2} \jb{n_3})^s \jb{n}^{3s}}  \\
		\times \|g_2(n_1)\|_{L^2_\tau} \|g_1(-n_2)\|_{L^2_\tau} \|g_1(n_3)\|_{L^2_\tau} \|g_1(n)\|_{L^2_\tau}^3 
		\lesssim \|g_1\|_{\l^p_n L^2_\tau}^5 \|g_2\|_{\l^p_n L^2_\tau},
		\label{resonant0}
	\end{multline}
	since $\|g_1\|_{\l^p_n L^2_\tau} \lesssim \|u\|_{X^{s,\frac12}_{p,2}}$ and $\|g_2\|_{\l^p_n L^2_\tau} = \|u\|_{X^{s,\frac12-}_{p,2}}$, from Lemma~\ref{lm:cutoff}.

	\smallskip
	\noi\textbf{\underline{Case 5: $|n_1| \sim |n_2| \sim |n_3|$}}\\
	Since $|\Phi(\nbar)|\ges |n_3| \lambda_1 \lambda_2$, for $\ld_j=|n-n'_j|$, $j=1,2$, and $n'_1, n'_2 \in \{n_1,n_2,n_3\}$ distinct, we have the following
	\begin{align*}
		\frac{|n|^{2+} |n_3|}{|\Phi(\nbar)|^\frac32 ( \jb{n_1} \jb{n_2} \jb{n_3})^{\frac14+} \jb{n}^{\frac34+} } \lesssim \frac{1}{\jb{n-n'_1}^\frac32 \jb{n-n'_2}^\frac32}.
	\end{align*}
	Then, since $|n| \lesssim |n_j|$, $j=1,2,3$, using Holder's inequality gives 
	\begin{align*}
		\text{LHS of }\eqref{resonant0} & \lesssim  \Bigg( \sum_{n,n_1',n_2'} \frac{\|g_1(n)\|_{L^2_\tau}^{3p'}}{\jb{n-n'_1}^{1+}\jb{n-n_2'}^{1+} \jb{n}^{6(s-\frac14)p'}}  \Bigg)^\frac{1}{p'} \|g_1\|_{\l^p_n L^2_\tau}^2 \|g_2\|_{\l^p_n L^2_\tau}\\
		& \lesssim \Bigg( \sum_n \frac{\|g_1(n)\|_{L^2_\tau}^{3p'}}{\jb{n}^{6(s-\frac14)p'}} \Bigg)^\frac{1}{p'} \|g_1\|_{\l^p_n L^2_\tau}^2 \|g_2\|_{\l^p_n L^2_\tau}^2 
	\end{align*}
	where the last inequality follows if $1\leq p <\infty$ and $s>\max\big(\frac{5}{12} - \frac{2}{3p},\frac14\big)$.
	
	\smallskip
	\noi\textbf{\underline{Case 6: $|n_3|^\frac12 \gg |n_1|, |n_2|$}}\\
	Since $|\Phi(\nbar)| \sim |n_3|^2 |n_1+n_2|$ and $|n_3|\sim|n| \gg |n_1|, |n_2|$, we have
	\begin{align*}
		\frac{|n|^{2+} |n_3|}{|\Phi(\nbar)|^{\frac32}} \lesssim \frac{|n|^{0+}}{\jb{n_1+n_2}^\frac32}.
	\end{align*}
	Using Holder's inequality, it follows that 
	\begin{align*}
		\text{LHS of } \eqref{resonant0}& \lesssim  \Bigg( \sum_{n_1,n_2,n} \frac{\|g_1(n)\|_{L^2_\tau}^{3p'} }{\jb{n_1+n_2}^{1+} \jb{n_1}^{sp'} \jb{n_2}^{sp'}  \jb{n}^{4sp'-}} \Bigg)^\frac{1}{p'} \|g_1\|_{\l^p_n L^2_\tau}^2 \|g_2\|_{\l^p_n L^2_\tau}  \\
		& \lesssim \Bigg( \sum_{n_1} \frac{1}{\jb{n_1}^{2sp'-}} \sum_n \frac{\|g_1(n)\|_{L^2_\tau}^{3p'} }{\jb{n}^{4sp'-}  } \Bigg)^\frac{1}{p'} \|g_1\|_{\l^p_n L^2_\tau}^2 \|g_2\|_{\l^p_n L^2_\tau}
	\end{align*}
	and the estimate follows if $1\leq p <\infty$ and $s>\frac12 - \frac{1}{2p}$.

	\smallskip
	\noi$\bullet$ \textbf{Estimate for $\mathcal{NR}_0$, $\mathcal{NR}_3$}\\
	We will omit the estimate for $\mathcal{NR}_3$ and focus on $\mathcal{NR}_0$. Let the following denote the modulations of the 6 factors
	\begin{align*}
		\s_j &= \tau_j-n_j^3 , \ j=1,2,3, \\
		\quad \s_4 &= \tau_4 - m_1^3 \quad \s_5  = \tau_5 - m_2^3, \quad
		\s_6  = \tau_6 -m_3^3,
	\end{align*}
	which implies that $\s_1 + \s_2 + \s_3 + \s_4 + \s_5 + \s_6 = \Phi(\nbar) + \Phi(\mbar)$. Thus, we will consider two regions:
	\begin{align}
		|\Phi(\mbar)| &\lesssim |\Phi(\nbar) + \Phi(\mbar)|, \label{resonant_good}\\
		|\Phi(\mbar)| &\gg |\Phi(\nbar) + \Phi(\mbar)|. \label{resonant_bad}
	\end{align} If \eqref{resonant_good} holds, we can use the largest modulation to gain a power of $|\Phi(\mbar)|^\frac12$. For \eqref{resonant_bad}, we have no gain from the largest modulation so we will use Strichartz estimates and the fact that $|\Phi(\nbar)| \sim |\Phi(\mbar)|$.
	Note that we can rewrite $\mathcal{NR}_0$ as follows
	\begin{multline*}
		\mathcal{NR}_0  
		=\Im\intt_{\tau_1 + \ldots + \tau_6 = 0} \sum_{|n|>N} \sum_{ \nbar\in\Lambda(n)} \sum_{\mbar\in\Lambda(-n)}\frac{n n_3 m_3}{\Phi(\nbar)} \ft{u}(\tau_1,n_1)\conj{ \Ft\big(\chi_{[0,t]} u\big)}(-\tau_2,-n_2) \\\times \ft{u}(\tau_3,n_3) 
		\conj{\ft{u}}(-\tau_4,-m_1) \ft{u}(\tau_5,m_2) \conj{\ft{u}}(-\tau_6,-m_3) \ d \tau_1 \cdots d\tau_5.
	\end{multline*}
	
	Consider the case \eqref{resonant_good} and proceed as in the estimate for $\mathcal{R}_0$. Assuming that we can associate the time cut-off with the first factor, we have
	\begin{multline}
		|\mathcal{NR}_0|  \lesssim \frac{1}{N^{0+}} \sum_{|n|>N} \sum_{\nbar\in\Lambda(n)} \sum_{\mbar\in\Lambda(-n)} \frac{|n|^{1+} |n_3m_3|}{|\Phi(\nbar)| | \Phi(\mbar)|^\frac12 \prod_{j=1}^3 \jb{n_j}^s \jb{m_j}^s} 
		\|g_2(n_1)\|_{L^2_\tau} \\ \times\|g_1(-n_2) \|_{L^2_\tau} \|g_1(n_3)\|_{L^2_\tau} \|g_1(-m_1)\|_{L^2_\tau} \|g_1(m_2)\|_{L^2_\tau} \|g_1(-m_3)\|_{L^2_\tau}.\label{nr0}
	\end{multline}
	For simplicity, we can apply Lemma~\ref{lm:cutoff} to obtain $\|g_2\|_{L^2_\tau} \lesssim \|g_1\|_{L^2_\tau}$.
	In order to control the multiplier in \eqref{nr0}, we must take into account the value of $\Phi(\mbar)$ and the relation between the frequencies of the first generation $n_1,n_2,n_3$.
	
	\smallskip
	\noi\textbf{\underline{Case 5 and \eqref{resonant_good}: $|n_1| \sim |n_2| \sim |n_3|$}}\\
	If $|m_1|\sim|m_2|\sim|m_3|$ and $|\Phi(\mbar)| \ges |m_3| |n+m'_1| |n+m'_2|$, for some distinct $m'_1,m'_2\in\{m_1,m_2,m_3\}$, we have
	\begin{align*}
		\frac{|n|^{1+} |n_3m_3|}{|\Phi(\nbar)| |\Phi(\mbar)|^\frac12} \lesssim \frac{|n_1n_2n_3|^{\frac13+} |m_1m_2|^{\frac14+}}{\jb{n-n'_1} \jb{n-n'_2} \jb{n+m'_1}^{\frac12+} \jb{n+m'_2}^{\frac12+}},
	\end{align*}
	for some distinct $n'_1,n'_2\in\{n_1,n_2,n_3\}$.
	Using H\"older's inequality, we get
	\begin{align*}
		\eqref{nr0}& \lesssim 
		\Bigg\| \sum_{\nbar\in\Lambda(n)} \frac{\|g_1(n_1)\|_{L^2_\tau} \|g_1(-n_2)\|_{L^2_\tau} \|g_1(n_3) \|_{L^2_\tau} }{\jb{n-n'_1} \jb{n-n'_2} (\jb{n_1} \jb{n_2} \jb{n_3})^{s-\frac13-}} \Bigg\|_{\l^2_n} \\
		& \qquad \times \Bigg\| \sum_{\mbar\in\Lambda(-n)} \frac{\|g_1(-m_1) \|_{L^2_\tau} \|g_1(m_2)\|_{L^2_\tau} \|g_1(-m_3)\|_{L^2_\tau}}{ \jb{n+m'_1}^{\frac12+} \jb{n+m'_2}^{\frac12+} (\jb{m_1} \jb{m_2} \jb{m_3})^{s-\frac13-}}  \Bigg\|_{\l^2_n} \\
		& \lesssim \!\sup_n \Bigg(\sum_{\substack{n'_1, n'_2, \\m'_1, m'_2}} \frac{1}{\jb{n-n'_1}^{1+} \jb{n-n'_2}^{1+}\jb{n+m'_1}^{1+} \jb{n+m'_2}^{1+}} \Bigg)^\frac12
		\Bigg\|\frac{g_1}{\jb{n}^{s-\frac13-}}\Bigg\|^6_{\l^2_n L^2_\tau} \\
		& \lesssim \|g_1\|^6_{\l^p_n L^2_\tau} = \|u\|^6_{X^{s,\frac12}_{p,2}},
	\end{align*}
	for $1\leq p <\infty$ and $s>\max\big( \frac13,\frac56-\frac1p\big)$.
	In the remaining regions of frequency space for $m_1,m_2,m_3$, we have $|\Phi(\mbar)| \ges |m_{\max}|^2 \ld'$, for $\ld' \in \{|m_{\max} + m_{\med}|, |m_{\med} + m_{\min}| \}$. Thus,
	\begin{align*}
		\frac{|n|^{1+} |n_3m_3|}{|\Phi(\nbar)| |\Phi(\mbar)|^{\frac12}} \lesssim \frac{|n_1n_2n_3|^{\frac13+}}{\jb{n-n'_1} \jb{n-n'_2} \jb{\ld'}^{\frac12}}.
	\end{align*}
	Since $(\jb{m_{\max}}\jb{m_{\med}})^{-\frac13+} \lesssim \jb{m_{\med}}^{-\frac23-}$, we can proceed as in the previous case, with $\jb{\ld'}^{\frac12+} \jb{m_{\med}}^{\frac23+}$ instead of $\jb{n+m'_1}^{\frac12+} \jb{n+m'_2}^{\frac12+}$.
	
	\smallskip
	
	\noi\textbf{\underline{Case 6 and \eqref{resonant_good}: $|n_{\max}|^2 \lesssim |\Phi(\nbar)|$ }}\\
	Since we have
	\begin{align*}
		\frac{|n|^{1+} |n_3|}{|\Phi(\nbar)| } \lesssim \frac{|n_1n_2|^{\frac14+} }{\jb{n_1+n_2} \jb{n_{\min}}^{\frac12+} },
	\end{align*}
	we can follow the same argument in the previous case, substituting $\jb{n-n'_1} \jb{n-n'_2}$ by $\jb{n_1+n_2} \jb{n_{\min}}^{\frac12+}$.
	
	\medskip
	
	Now, we must consider \eqref{resonant_bad}. Since we have $|\Phi(\nbar)| \sim |\Phi(\mbar)|$, we focus on estimating the following multiplier
	\begin{equation}
		\frac{|n|^{1+}|n_3m_3|}{|\Phi(\nbar)|^\al |\Phi(\mbar)|^{1-\al}}, \label{nr0_multiplier}
	\end{equation}
	for some $0\leq \al \leq1$.
	
	\smallskip
	
	\noi\textbf{\underline{Case 5 and \eqref{resonant_bad}: $|n_1|\sim|n_2|\sim|n_3|$ }}\\
	Choosing $\al=0$, we have
	\begin{align*}\eqref{nr0_multiplier} \lesssim 
		\begin{cases}
			|n_1n_2n_3m_1m_2m_3|^{\frac13+}, & \text{ if } \quad |m_1|\sim|m_2|\sim|m_3| \\
			|n_1n_2n_3|^{\frac13+}, & \text{ if } \quad |\Phi(\mbar)| \ges |m_{\max}|^2
		\end{cases}.
	\end{align*}
	Let $\ft{h}_1(\tau,n) = \jb{n}^{\frac13+} \Ft_{t,x}\big(\chi_{[0,t]} u\big)(\tau,n)$, ${\ft{h}_2(\tau,n) = \jb{n}^{\frac13+} |\ft{u}(\tau,n)|}$ and note that we can associate the cut-off with any factor. Using H\"older's inequality, the Strichartz estimate \eqref{L6new} and Lemma~\ref{lm:cutoff}, we get
	\begin{align*}
		|\mathcal{NR}_0| &\lesssim  \frac{1}{N^{0+}} \| h_1 h_2^5\|_{L^1_{t,x}} \lesssim \frac{1}{N^{0+}} \|h_1\|_{L^6_{t,x}} \|h_2\|_{L^6_{t,x}}^5 
		\lesssim  \frac{1}{N^{0+}} \|u\|_{X^{s,\frac12}_{p,2}}^6,
	\end{align*}
	for $1\leq p <\infty$ and $s>\max\big(\frac13, \frac56 - \frac1p\big)$.

	\smallskip
	\noi\textbf{\underline{Case 6 and \eqref{resonant_bad}: $|n_3|^\frac12\gg |n_1|, |n_2|$}}\\
	If $|m_1|\sim|m_2|\sim|m_3|$, choosing $\al =1$, gives $\eqref{nr0_multiplier} \lesssim |m_1m_2m_3|^{\frac13+}$ and the result follows from the previous case. Now, assume that $|\Phi(\mbar)|\sim|m_{\max}|^2 \ld'$ where $\ld' \in\{|m_{\max} + m_{\med}|, |m_{\med} + m_{\min}|\}$. We must consider a finer case separation for the second generation of frequencies. For $\al=0$, we can estimate the multiplier as follows
	\begin{equation*}
		\eqref{nr0_multiplier} \lesssim
		\begin{cases}
			|n_3m_3 \max(|m_1|, |m_2|)|^{\frac13+}, & \text{if} \quad |m_3| \lesssim \max(|m_1|, |m_2|) \\
			|n_3m_1m_2m_3|^{\frac13+}, & \text{if} \quad |m_3|^\frac12 \lesssim |m_{\min}| \leq |m_{\med}| \ll |m_3|\\
			|n_3m_3|^{\frac14+} , & \text{if} \quad |m_{\min}| \ll |m_3|^\frac12 \lesssim |m_{\med}| \ll |m_3|
		\end{cases}
	\end{equation*}
	and use the strategy in the previous case.
	
	It only remains to consider the case when $|m_3|^\frac12 \gg |m_1|, |m_2|$.
	Consider the following decomposition
	\begin{align*}
		\mathcal{NR}_0  & = \Im \int_0^t \sum_{|n|>N} \sum_{\substack{\nbar\in\Lambda(n)}}\frac{m_3}{n_1 + n_2} \sum_{\substack{\mbar\in\Lambda(-n)}}  \Bigg( \frac{nn_3}{(n_1+n_3)(n_2+n_3)} - 1 \Bigg) \\
		& \qquad \qquad \times  \ft{u}(n_1) \conj{\ft{u}}(-n_2) \ft{u}(n_3) \conj{\ft{u}}(-m_1) \ft{u}(m_2) \conj{\ft{u}}(-m_3) \, dt'  \\
		& \phantom{X}+ \Im \int_0^t \sum_{|n|>N} \sum_{\substack{\nbar\in\Lambda(n)}} \sum_{\substack{\mbar\in\Lambda(-n)}}  \frac{m_3}{n_1 + n_2}\ft{u}(n_1) \\
		& \qquad \qquad \times \conj{\ft{u}}(-n_2) \ft{u}(n_3) \conj{\ft{u}}(-m_1) \ft{u}(m_2) \conj{\ft{u}}(-m_3) \, dt' \\
		&=: \I_0 + \II_0.
	\end{align*}
	
	In order to estimate $\I_0$, note that $$nn_3 - (n_1+n_3)(n_2+n_3) = n_3^2 + (n_1+n_2)n_3 - n_1 n_2 - (n_1 + n_2)n_3 - n_3^2 = -n_1 n_2,$$ which implies that
	\begin{align*}
		\Bigg|\frac{nn_3}{(n_1+n_3)(n_2+n_3)} - 1 \Bigg| = \frac{|n_1 n_2|}{|(n_1+n_3)(n_2+n_3)|} \lesssim \frac{|n_3|}{|n_3|^2} \lesssim \frac{1}{|n_3|}.\end{align*}
	Hence, using H\"older's inequality and $L^6$-Strichartz estimates \eqref{L6new}, we have
	\begin{align*}
		|\I_0| \lesssim \frac{1}{N^{0+}} \| \chi_{[0,t]} u^6 \|_{L^1_{t,x}} 
		\lesssim \frac{1}{N^{0+}} \|\chi_{[0,t]} u \|_{X^{0+,\frac12-}_{2,2}} \|u\|^5_{X^{0+,\frac12-}_{2,2}}
		\lesssim \frac{1}{N^{0+}}  \|u\|^6_{X^{s,\frac12}_{p,2}}
	\end{align*}
	for $1\leq p <\infty$ and $s>\max\big(\frac12 - \frac1p,0)$.
	
	Now, we focus on estimating $\II_0$. First, assume that $n_3 + m_3 \neq 0$. Then, 
	\begin{multline*}
		|\Phi(\nbar) + \Phi(\mbar)| = |3(n_3+m_3)(n_1+n_3)(n_1+m_3) \\+ 3(n_2 + m_1)(n_2+m_2)(m_1+m_2)|
		\ges |n_3|^2,
	\end{multline*}
	since $|(n_3+m_3)(n_1+n_3)(n_1+m_3)| \ges |n_3|^2$ and $|(n_2 + m_1)(n_2+m_2)(m_1+m_2)| \ll |n_3|^\frac32$. Then, using the largest modulation, we have
	\begin{align*}
		\frac{|m_3|}{|\Phi(\nbar) + \Phi(\mbar)|^\frac12} \lesssim 1.
	\end{align*}
	Proceeding as in \eqref{nr0}, we first focus on estimating $\II_0$ with respect to time
	\begin{multline}
		|\II_0| \lesssim \frac{1}{N^{0+}} \sum_{|n|>N} \sum_{\nbar\in\Lambda(n)} \sum_{\mbar\in\Lambda(-n)} \frac{1 }{\jb{n_1+n_2}\prod_{j=1}^3(\jb{n_j} \jb{m_j})^{s-}} \|g_1(n_1)\|_{L^2_\tau}  \\ \times \|g_1(-n_2)\|_{L^2_\tau} \|g_1(n_3)\|_{L^2_\tau}  \|g_1(-m_1)\|_{L^2_\tau} \|g_1(m_2)\|_{L^2_\tau} \|g_1(-m_3)\|_{L^2_\tau}. \label{II0}
	\end{multline}
	The estimate follows from the approach in Case 5 and \eqref{resonant_good}, since
	\begin{align*}
		\frac{1}{\jb{n_1+n_2} \prod_j^3 (\jb{n_j} \jb{m_j})^{\frac13+}} \lesssim \frac{1}{\jb{n_1+n_2} \jb{n_{\min}}^{\frac12+} \jb{m_{\min}}^{\frac12+} \jb{m_{\med}}^{\frac12+} }.
	\end{align*}
	On the other hand, if $n_3 + m_3 = 0$, focus on the following quantity 
	\begin{multline*}
		\II_0 = \int_0^t \sum_{|n|>N} \sum_{\substack{\nbar\in\Lambda(n)\\|n_1|,|n_2|\ll |n_3|^\frac12}} \sum_{\substack{\mbar\in\Lambda(-n),\\|m_1|,|m_2|\ll |n_3|^\frac12}} \frac{-n_3}{n_1 + n_2}\\
		\times  \ft{u}(n_1) \conj{\ft{u}}(-n_2) \conj{\ft{u}}(-m_1) \ft{u}(-n_1-n_2-m_1) |\ft{u}(n_3)|^2 \, dt'.
	\end{multline*}
	In order to estimate this quantity we need further assumptions on the frequencies. Let $\eps>0$ denote the constant such that $|n_1|, |n_2|, |m_1|, |m_2| \leq \eps |n_3|^\frac12$. We will consider two distinct cases: (i) $|n_1 + n_2| > \eps^2 |n_3|^\frac12$; (ii) $|n_1 + n_2| \leq \eps^2 | n_3|^\frac12$. 
	
	If $|n_1 + n_2| > \eps^2 |n_3|^\frac12$, then
	\begin{align*}
		\frac{|n_3|}{|n_1+n_2|\jb{n_3}^{\frac12+}} \lesssim \frac{1}{N^{0+}}.
	\end{align*}
	For simplicity, assume that $|n_1| \leq |n_2|$ and $|m_1|\leq |m_2|$. Consequently, following a similar approach to \eqref{II0} to handle the time integral, with $h(\tau,n) = \jb{n}^{\frac13+} \jb{\tau-n^3}^{\frac12-} |\ft{u}(\tau,n)|$, and using H\"older's inequality, we obtain
	\begin{align*}
		|\mathcal\II_0| & \lesssim \frac{1}{N^{0+}}  \|h\|^2_{\l^2_n L^2_\tau} \\
		& \quad \times \sum_{\substack{n_1, n_2, m_1}} \frac{ \|h(n_1)\|_{L^2_\tau} \|h(-n_2)\|_{L^2_\tau} \|h(-m_1)\|_{L^2_\tau} \|h(-n_1-n_2-m_1)\|_{L^2_\tau}
		}{N^{0+}(\jb{n_1} \jb{n_2}\jb{m_1}\jb{n_1+n_2+m_1})^{\frac13+} (\jb{n_1}\jb{m_1})^{\frac16+} } \nonumber\\
		& \lesssim \frac{1}{N^{0+}} \Bigg( \sum_{n_1,n_2,m_1} \frac{\|h(-n_2)\|^2_{L^2_\tau}}{ \jb{n_1}^{1+} \jb{m_1}^{1+}} \Bigg)^\frac12 \|u\|^5_{X^{s,\frac12}_{p,2}} \nonumber\\
		& \lesssim  \frac{1}{N^{0+}}\|u\|^6_{X^{s,\frac12}_{p,2}}, 
	\end{align*}
	for $1\leq p <\infty$ and $s>\max\big(\frac13, \frac56-\frac1p\big)$.
	
	It remains to estimate the case when $|n_1+n_2| \leq \eps^2|n_3|^\frac12$. Under this assumption and $|n_j|\leq \eps |n_3|^\frac12$, $j=1,2$, it follows that  $|n_j| \leq \eps |n_3|^\frac12 - |n_1+n_2|$ or $\eps|n_3|^\frac12 - |n_1+n_2| < |n_j| < \eps |n_3|^\frac12$, $j=1,2$. For simplicity, let $|n_1| \leq |n_2|$ and $|m_1| \leq |n_1+n_2+m_2|$, as the result follows from an analogous approach for the remaining cases.
	We consider the following two regions of summation
	\begin{align*}
		H_1 &:= \big\{ (n_1,n_2,m_1): \ |n_1|, |m_1| < \eps |n_3|^\frac12 - |n_1+n_2| , \\
		&\qquad \qquad \qquad  |n_2|, |n_1+n_2+m_1| < \eps|n_3|^\frac12, \ |n_1+n_2| < \eps^2 |n_3|^\frac12 \big\}, \\
		H_2 &:= \big\{ (n_1,n_2,m_1): \ |n_1|, |n_2|, |m_1|, |n_1+n_2+m_1|\leq \eps |n_3|^\frac12, \\
		& \qquad\qquad\qquad |n_1| \text{ or } |m_1| \geq \eps|n_3|^\frac12 - |n_1+n_2|,  \ |n_1+ n_2| < \eps^2 |n_3|^\frac12  \big\}.
	\end{align*}
	We first consider the contribution restricted to the region $H_2$, when $|n_1| \geq \eps |n_3|^\frac12 - |n_1+n_2|$. Note that the following holds
	\begin{align*}
		|n_1| \geq \eps |n_3|^\frac12 - |n_1+n_2| \geq (\eps - \eps^2) |n_3|^\frac12. 
	\end{align*}
	Therefore, the multiplier can be controlled as follows
	\begin{align*}
		\frac{|n_3|}{|n_1+n_2| \jb{n_1}^{\frac13+} \jb{n_2}^{\frac13+} \jb{n_3}^{\frac23+}} \lesssim \frac{1}{N^{0+}|n_1+n_2|^{1+}} .
	\end{align*}
	The estimate follows the same approach as \eqref{II0}, for $s>\max\big(\frac13,\frac56-\frac1p\big)$, $1\leq p <\infty$.
	
	Now, consider the contribution localized on the region $H_1$, with the change of variables $n'_2 = n_1 + n_2$, 
	\begin{multline*}
		\int_0^t\sum_{\substack{ |n|>N, \\  |n'_2|<\eps^2|n-n'_2|^\frac12}}  \frac{n-n'_2}{n'_2} |\ft{u}(n-n'_2)|^2 \\
		\times  \Bigg( \Im \sum_{\substack{|n_1|, |m_1| \\<\eps |n-n'_2|^\frac12 - |n'_2|} }  \ft{u}(n_1) \conj{\ft{u}}(n_1-n'_2) \conj{\ft{u}}(-m_1) \ft{u}(-n'_2-m_1) \Bigg) \ dt'.
	\end{multline*}
	Use $J$ to denote the two inner sums. We can decompose $J$ as follows
	\begin{align*}
		J & = \Im \bigg(\sum_{0<n_1,m_1<\eps|n-n'_2|^\frac12 - |n'_2|} \ft{u}(n_1) \conj{\ft{u}}(n_1-n'_2) \conj{\ft{u}}(-m_1) \ft{u}(-n'_2-m_1) \\
		& \phantom{xxxxxx}+ \sum_{0<n_1,m_1<\eps|n-n'_2|^\frac12 - |n'_2|} \ft{u}(-n_1) \conj{\ft{u}}(-n_1-n'_2) \conj{\ft{u}}(-m_1) \ft{u}(-n'_2-m_1) \\
		& \phantom{xxxxxx}+ \sum_{0<n_1,m_1<\eps|n-n'_2|^\frac12 - |n'_2|} \ft{u}(n_1) \conj{\ft{u}}(n_1-n'_2) \conj{\ft{u}}(m_1) \ft{u}(-n'_2+m_1) \\
		&\phantom{xxxxxx}+\sum_{0<n_1,m_1<\eps|n-n'_2|^\frac12 - |n'_2|} \ft{u}(-n_1) \conj{\ft{u}}(-n_1-n'_2) \conj{\ft{u}}(m_1) \ft{u}(-n'_2+m_1) \\
		& \phantom{xxxxxx}+\sum_{0<|n_1|<\eps|n-n'_2|^\frac12 - |n'_2|} \ft{u}(0) \conj{\ft{u}}(-n'_2) \conj{\ft{u}}(-n_1) \ft{u}(-n'_2-n_1) \\
		& \phantom{xxxxxx}+  \sum_{0<|n_1|<\eps|n-n'_2|^\frac12 - |n'_2|} \ft{u}(-n_1) \conj{\ft{u}}(-n_1-n'_2) \conj{\ft{u}}(0) \ft{u}(-n'_2) \\
		& \phantom{xxxxxx}+  \ft{u}(0) \conj{\ft{u}}(-n'_2) \conj{\ft{u}}(0) \ft{u}(-n'_2) \bigg)\\
		& = \Im \sum_{\substack{0<|n_1|,|m_1|<\eps|n-n'_2|^\frac12 - |n'_2|\\n_1m_1>0}} \ft{u}(n_1) \conj{\ft{u}}(n_1-n'_2) \conj{\ft{u}}(m_1) \ft{u}(-n'_2+m_1) \\
		& = \Im \bigg( \frac12 \sum_{\substack{0<|n_1|,|m_1|<\eps|n-n'_2|^\frac12 - |n'_2|\\n_1m_1>0, n_1\neq m_1}} \ft{u}(n_1) \conj{\ft{u}}(n_1-n'_2) \conj{\ft{u}}(m_1) \ft{u}(-n'_2+m_1) \\
		& \phantom{xxxxxx}+  \frac12 \sum_{\substack{0<|n_1|,|m_1|<\eps|n-n'_2|^\frac12 - |n'_2|\\n_1m_1>0, n_1\neq m_1}} \ft{u}(m_1) \conj{\ft{u}}(m_1-n'_2) \conj{\ft{u}}(n_1) \ft{u}(-n'_2+n_1) \\
		& \phantom{xxxxxx}- \sum_{\substack{0<|n_1|<\eps|n-n'_2|^\frac12 - |n'_2|}} \ft{u}(n_1) \conj{\ft{u}}(n_1-n'_2) \conj{\ft{u}}(n_1) \ft{u}(-n'_2+n_1) \bigg) =0 .
	\end{align*}
	This completes the estimate for the contribution $\mathcal{NR}_0$.
	
	\medskip
	\noi$\bullet$ \textbf{Estimate for $\mathcal{NR}_1, \mathcal{NR}_2$}\\
	In order to control the contributions $\mathcal{NR}_1, \mathcal{NR}_2$, we will follow a similar approach to that of $\mathcal{NR}_0$. Most cases follow an analogous approach, but the estimate is significantly different in Case 6, when $|\Phi(\mbar)| \ges |m_{\max}|^2$ and \eqref{resonant_bad} hold.
	
	In this case, we cannot use the maximum modulation to help estimate the multiplier. However, we can use the fact that $|\Phi(\nbar)| \sim |\Phi(\mbar)|$ to obtain the following
	\begin{align}
		\frac{|nn_3m_3|}{|\Phi(\nbar)|^\al |\Phi(\mbar)|^{1-\al}} \lesssim \frac{|n|^{1+} |n_3m_3|}{N^{0+}|n_3|^{2\al} |m_{\max}|^{2(1-\al)}},\label{sep_aux2}
	\end{align}
	for some $0\leq \al \leq 1$.
	Estimating this multiplier requires more care than for the $\mathcal{NR}_0$ contribution since we cannot directly compare the sizes of $|n|, |n_3|$ and $|m_{\max}|$. We can estimate the multiplier as follows
	\begin{align*}
		\eqref{sep_aux2} \lesssim
		\begin{cases}
			|n_3m_1m_2m_3|^{\frac14+}, & \text{if} \quad |m_3| \lesssim |m_1|, |m_2| \text{ and } \al = \frac78,\\
			|nn_3m_3m_{\max}|^{\frac14+}, & \text{if} \quad |m_{\min}| \ll |m_3| \lesssim |m_3| \text{ and } \al=\frac34 
		\end{cases}
	\end{align*}
	and following the previous arguments, using H\"older's inequality and the $L^6$-Strichartz estimate \eqref{L6new}. If $|m_3| \gg |m_1|, |m_2|$, then $|\Phi(\nbar)| \gg |\Phi(\mbar)|$ and $|\Phi(\nbar) + \Phi(\mbar)| \sim |\Phi(\nbar)|$, which contradicts our assumptions.

\end{proof}

\section{A priori estimate and global well-posedness} \label{sec:gwp}
In this section, we focus on showing the global well-posedness of the real-valued mKdV1 equation \eqref{renorm}. Note that the same argument can be used to extend solutions of mKdV2 \eqref{renorm2} globally-in-time.

The following result from \cite{OH1} is essential to extend local-in-time solutions to global ones.

\begin{proposition}\label{prop:apriori}
	Let $2\leq p<\infty$ and $0<s<1-\frac1p$. There exists $C=C(p)>0$ such that 
	\begin{align}
		\|u(t)\|_{\FL^{s,p}} \leq C (1 + \|u(0)\|_{\FL^{s,p}})^{\frac{p}{2}-1} \|u(0)\|_{\FL^{s,p}},		
		\label{apriori}
	\end{align}
	for any smooth solutions $u$ to the complex-valued mKdV1 equation \eqref{renorm}, for any $t\in\R$.
\end{proposition}

When $2\leq p < \infty$ and $\frac34-\frac1p<s<1 - \frac1p$, the global well-posedness immediately follows from the local well-posedness in Theorem \ref{th:lwpreal} and the global-in-time bound \eqref{apriori} in Proposition~\ref{prop:apriori}, by iterating the local argument.
However, we want to remove the upper bound on $s$, using a persistence-of-regularity argument. Before proving Theorem \ref{th:gwp}, we need to modify the nonlinear estimate in Section \ref{sec:nonlinear} accordingly. 
\begin{proposition}\label{prop:nonlinear_new}
	Let $(s,p)$ satisfy $1\leq p <\infty$ and $s\geq \s(p)$, where
	\begin{equation}\s(p)= \begin{cases}
			\frac12 , & 1\leq p <4 \\
			\frac34-\frac1p + , & p \geq 4
		\end{cases}.
		\label{sigma}
	\end{equation}
	Then, the following estimates hold
	\begin{align*}
		\|\N^*(u)\|_{Z^{s,-\frac12}_p(T)} &\lesssim T^\delta  \|u\|_{X^{s,\frac12}_{p,2}(T)} \|u\|_{X^{\s(p),\frac12}_{p,2}(T)}^2  &&\text{if } \ 2\leq p <\infty, \\
		\|\N^*(u)\|_{Z^{s, -\frac12}_p(T)} & \lesssim T^\delta \|u\|_{X^{s,\frac12}_{p,2}(T)} \|u\|^2_{X^{\frac12, \frac12}_{2, 2}(T)}  && \text{if } \ 1\leq p <2, \nonumber
	\end{align*}
	for some $0<\delta\ll1$ and any $0<T \leq 1$.
\end{proposition}
\begin{proof}
	The proof follows a similar argument to that of Proposition~\ref{prop:nonlinear}. We will only focus on Case 1.1 and Part 3, pointing out where the proof diverges.
	
	Following the argument in Case 1.1, we arrive at the following estimate, instead of \eqref{nonlinear_main},
	\begin{align}
		\|\mathcal{NR}(u,u,u)\|_{X^{s, -\frac12}_{p,2}} \lesssim \bigg\| \sum_{\nbar\in\Lambda(n)} \frac{\jb{n}^s |n_{\max}|}{|\Phi(\nbar)|^\frac12  \prod_{j=1}^3 \jb{n_j}^{\al_j}} \prod_{j=1}^3 \|f_j(n_j)\|_{L^2_\tau} \bigg\|_{\l^p_n}, \label{new_nonlinear_aux1}
	\end{align}
	where $f_j(\tau,n) = \jb{n}^{\al_j} \jb{\tau-n^3}^{\frac12-\nu} |\ft{u}(\tau,n)|$ for some $\al_j\geq 0$, $j=1,2,3$, and $n_{\max} = \max_{j=1,2,3}|n_j|$. For simplicity, assume $|n_3| \leq |n_2| \leq |n_1|$. 
	If $2\leq p <\infty$, let $\al_1 = s$, $\al_2 = \al_3 = \s(p)$. Using H\"older's inequality, the estimate follows once we prove that
	$$\tilde{J}_1(n) = \sum_{\nbar\in\Lambda(n)} \bigg( \frac{\jb{n}^s |n_1|}{|\Phi(\nbar)|^\frac12 \jb{n_1}^s \jb{n_2}^{\s(p)} \jb{n_3}^{\s(p)}} \bigg)^{p'} \lesssim 1,$$
	which follows from the arguments used to estimate \eqref{J1aux}.
	If $1\leq p <2$, let $\al_1 = s$, $\al_2 = \al_3 > \frac14$ and apply H\"older's inequality to obtain
	\begin{align*}
		\text{RHS of } \eqref{new_nonlinear_aux1}& \lesssim \sup_n \bigg( \sum_{\nbar\in\Lambda(n)} \frac{\jb{n}^s |n_1|}{|\Phi(\nbar)|^\frac12 \jb{n_1}^s \jb{n_2}^{\frac14+} \jb{n_3}^{\frac14+}} \prod_{j=2}^3 \|f_j(n_j)\|_{L^2_\tau} \bigg) \|f_1\|_{\l^p_n}.
	\end{align*}
	Using Cauchy-Schwarz inequality, the estimate follows once we prove that 
	$$\tilde{J}_2(n) = \sum_{\nbar\in\Lambda(n)} \bigg( \frac{\jb{n}^s |n_1|}{|\Phi(\nbar)|^\frac12 \jb{n_1}^s \jb{n_2}^{\frac14+} \jb{n_3}^{\frac14+}} \bigg)^{2} \lesssim 1.$$
	This follows by previously seen arguments.
	
	Considering the resonant contribution, following the arguments in Part 3, we have
	\begin{align*}
		\|\mathcal{R}(u)\|_{X^{s,-\frac12 + \nu}_{p,2}} & \lesssim \bigg\| \jb{n}^s |n|  \big\| \jb{\tau-n^3}^{\frac12-\nu} \ft{u}(\tau,n) \big\|_{L^2_\tau}^3 \bigg\|_{\l^p_n} \\
		& \lesssim \| u \|_{X^{s,\frac12-\nu}_{p,2}} \|u\|_{X^{\frac12, \frac12-\nu}_{\infty, 2}}^2 
	\end{align*}
	and the intended estimate follows from the embeddings of the $\l^p$-spaces.
	
\end{proof}

It is now possible to prove Theorem \ref{th:gwp}.
\begin{proof}[Proof of Theorem \ref{th:gwp}]
	If $2\leq p <\infty$ and $\max(\frac12,\frac34 - \frac1p) < s< 1 - \frac1p$, the result follows from Theorem \ref{th:lwp} and the a priori bound \eqref{apriori} in Proposition \ref{prop:apriori}.
	
	Now, consider the case when $2\leq p <\infty$ and $s\geq 1 - \frac1p$. Then, $u_0 \in \FL^{s,p}(\T) \subset \FL^{\s(p),p}(\T)$ with $\s(p)$ as defined in \eqref{sigma} and there exists a unique global solution $u\in C\big(\R; \FL^{\s(p),p}(\T)\big)$. Using the a priori bound in Proposition~\ref{prop:apriori} when running a contraction mapping argument in $Z^{\s(p),\frac12}_p(I)$, for any interval $I$ of length $T>0$, imposes a local time of existence 
	\begin{align}
		T \sim (1 + \|u_0\|_{\FL^{\s(p),p}})^{-\theta}>0, \label{gwpT}
	\end{align}
	for the resulting solution, for some $\theta>0$. Moreover, by choosing $I = [t_0, t_0 + T]$, we get
	\begin{align}
		\|u\|_{Z^{\s(p),\frac12}_q(I)} \leq C \|u(t_0)\|_{\FL^{\s(p),p}}, \label{aux_global}
	\end{align}
	for some $C>0$. Note that by using the a priori bound, the bounds \eqref{gwpT} and \eqref{aux_global} hold uniformly in $t_0$. Using Proposition~\ref{prop:nonlinear_new} and \eqref{aux_global}, it follow that
	\begin{align*}
		\|u\|_{Z^{s,\frac12}_p (I)} & \leq C_1 \|u(t_0)\|_{\FL^{s,p}} + C_2 T^\delta \|u(t_0)\|^2_{\FL^{\s(p),p}} \|u\|_{X^{s,\frac12}_{p,2}(I)}
	\end{align*}
	for constants $C_1,C_2>0$. Using the a priori bound, we have
	$$C_3 T^\delta \|u(t_0)\|_{\FL^{\s(p),p}}^2 \leq C_4 T^\delta \big( 1 + \|u_0\|_{\FL^{\s(p),p}} \big)^{p-2} \| u_0\|_{\FL^{\s(p),p}}^2 \leq \frac12,$$
	where the last inequality holds by possibly refining the choice of $\theta$ in \eqref{gwpT}. Using the embedding $Z^{s,\frac12}_p (I) \embeds C(I;\FL^{s,p}(\T))$, it follows that
	\begin{align*}
		\sup_{t\in I}\|u(t)\|_{\FL^{s,p}} \leq 2 C_1 \|u(t_0)\|_{\FL^{s,p}},
	\end{align*}
	Iterating this argument, we obtain 
	\begin{align*}
		\sup_{t\in[-T^*,T^*]} \|u(t)\|_{\FL^{s,p}} \leq (2C_1)^{\big(1+\|u_0\|_{\FL^{s',q}}\big)^\theta T^*} \|u_0\|_{\FL^{s,p}},
	\end{align*}
	for any $T^*>0$. This shows the global well-posedness of \eqref{renorm} in $\FL^{s,p}(\T)$ for $2\leq p <\infty$ and $s\geq 1-\frac1p$. 
	
	Lastly, if $1\leq p <2$ and $s\geq \frac12$, note that $u_0\in\FL^{s,p}(\T) \subset H^\frac12(\T)$, thus we can follow the persistence of regularity argument, using the a priori bound in Proposition~\ref{prop:apriori} for solutions in $C(I;H^\frac12(\T))$ and the nonlinear estimate in Proposition~\ref{prop:nonlinear_new}.
	
\end{proof}

\appendix 
\section{Mild ill-posedness in $\FL^{s,p}(\T)$ for $s<\frac12$} \label{sec:ap}
In the following, we show the failure of uniform continuity of the data-to-solution map of the complex-valued mKdV \eqref{mkdv} on bounded sets of $\FL^{s,p}(\T)$, for $1\leq p \leq \infty$ and $s<\frac12$. 
The proof follows an argument by Burq-G\'erard-Tzvetkov \cite{BGT} and Christ-Colliander-Tao \cite{CCT03}.
\begin{lemma}
	Let $s<\frac12$ and $1\leq p \leq \infty$. There exist two sequences $\{u_{0n}\}_{n\in\NB}$, $\{\tilde{u}_{0n}\}_{n\in\NB}$ in $C^\infty(\T)$ satisfying the following conditions:
	\begin{enumerate}
		\item $\{u_{0n}\}_{n\in\NB}$, $\{\tilde{u}_{0n}\}_{n\in\NB}$ are uniformly bounded in $\FL^{s,p}(\T)$;
		
		\item $\lim\limits_{n\to\infty} \| u_{0n} - \tilde{u}_{0n} \|_{\FL^{s,p}} =0$;
		
		\item Let $u_n$, $\tilde{u}_n$ be the solutions to \eqref{mkdv} with initial data $u_{0n}$, $\tilde{u}_{0n}$, respectively. Then, there exists $C>0$ such that 
		\begin{align*}
			\underset{n\to\infty}{\liminf} \sup_{t\in[-T,T]}\| u_n(t) - \tilde{u}_n(t) \|_{\FL^{s,p}} \geq C , 
		\end{align*}
		for any $T>0$. 
	\end{enumerate}
	
\end{lemma}

\begin{proof}
	Let $N\in\NB$ and $a\in\C$. Define $u^{N,a}$ as follows
	\begin{align*}
		u^{N,a}(t,x) : = N^{-s} a e^{i(Nx + N^3t \pm |a|^2 N^{1-2s}t )},
	\end{align*}
	a smooth solution to \eqref{mkdv}.
	Given $n\in\NB$, let $u_{0n} = u^{N_n,1}(0)$ and $\tilde{u}_{0n} = u^{N_n, 1+\frac1n}(0)$, for some $N_n \in\NB$ to be chosen later. Then, 
	\begin{align*}
		\|u_{0n}\|_{\FL^{s,p}}, \|\tilde{u}_{0n} \|_{\FL^{s,p}} \lesssim 1, \\
	\end{align*}
	uniformly in $n\in\NB$. Moreover,
	\begin{align*}
		\|u_{0n} - \tilde{u}_{0n} \|_{\FL^{s,p}} \sim \frac1n.
	\end{align*}
	Let $u_n = u^{N_n,1}$, $\tilde{u}_n = u^{N_n,1+\frac1n}$ be the solutions corresponding to initial data $u_{0n}$, $\tilde{u}_{0n}$, respectively. 
	Now, considering the difference between the two solutions at time $t\in\R$, we have
	\begin{align*}
		\|u_n(t) - \tilde{u}_n(t) \|_{\FL^{s,p}} \sim \bigg|e^{\pm N^{1-2s} \big(1-(1+\frac1n)^2\big)t} - \Big(1+\frac1n \Big) \bigg|.
	\end{align*}
	Therefore, the solutions have opposite phases at time $t_n>0$ defined as follows
	\[ t_n = \frac{\pi N_n^{2s-1}}{\big( 1 + \frac1n\big)^2 - 1}. \]
	Since $s<\frac12$, we can choose $N_n$ large enough, such that $t_n \leq \frac1n$. Consequently, we have
	\begin{align*}
		\|u_n(t_n) - \tilde{u}_n(t_n) \|_{\FL^{s,p}} \sim 2 + \frac1n \geq 2.
	\end{align*}
	Since $t_n\to 0$ as $n\to\infty$, the functions constructed satisfy the intended conditions and the lemma follows.
	
\end{proof}

\section*{Acknowledgments} 
A.C. would like to thank her advisor, Tadahiro Oh, for suggesting the problem and for his continuous support throughout the work. The author would like to thank Justin Forlano and Tristan Robert for helpful discussions and encouragement. A.C.~acknowledges support from Tadahiro Oh’s ERC Starting Grant (no.~637995 ProbDynDispEq) and the Maxwell Institute Graduate School in Analysis and its Applications, a Centre for Doctoral Training funded by the UK Engineering and Physical Sciences Research Council (grant EP/L016508/01), the Scottish Funding Council, Heriot-Watt University and the University of Edinburgh. The author is also grateful to the anonymous referees for their helpful comments.


\end{document}